\documentclass[a4paper,11pt, reqno]{amsart} 
\usepackage{mathptmx}
\usepackage{tabularx}
\usepackage{booktabs}
\usepackage{multirow}
\usepackage{amsmath}
\usepackage[T1]{fontenc}
\usepackage{lmodern}
\usepackage[utf8]{inputenc}
\usepackage{mathtools} 

\usepackage{amssymb}
\usepackage{amsfonts}
\usepackage{fancyhdr}
\usepackage{graphicx}
\usepackage{hyperref}
\usepackage{enumerate}
\usepackage{amsthm}

\usepackage[margin=1in]{geometry}

\newtheorem{theorem}{Theorem}[section]
\newtheorem{lemma}[theorem]{Lemma}
\newtheorem{corollary}[theorem]{Corollary}
\newtheorem{proposition}[theorem]{Proposition}

\newtheorem{definition}[theorem]{Definition}
\newtheorem{remark}[theorem]{Remark}

\newcommand{\set}[1]{\left\{#1\right\}}
\newcommand{\R}{\mathbb R}
\newcommand{\Z}{\mathbb Z}

\newcommand{\si}{\sigma}

\newcommand{\la}{\lambda}
\newcommand{\C}{\mathbb C }

\newcommand{\pa}{\partial }

\newcommand{\s}{\mathbf S}
\newcommand{\om}{ \omega}

\newcommand{\D}{\Delta}
\newcommand{\ap}{\alpha}
\newcommand{\bt}{\beta}

\newcommand{\gm}{\gamma}
\newcommand{\Na}{\mathbb N}

\newcommand{\re}{\operatorname{Re}}
\newcommand{\im}{\operatorname{Im}}

\newcommand{\ds}{\displaystyle}
\newcommand{\sar}{S^a_{rad}}
\newcommand{\sbr}{S^b_{rad}}
\newcommand{\sa}{S^a}
\usepackage{tikz}
\usetikzlibrary{calc}

\title[Characterization of eigenfunctions of Laplacian having exponential growth]
{Characterization of eigenfunctions of Laplacian having exponential growth using Fourier multipliers}

\author[Basil Paul, Pradeep B]{Basil Paul, Pradeep Boggarapu}
\address[Basil Paul]{Department of Mathematics\\
		BITS Pilani K K Birla Goa Campus\\
		Zuarinagar, South Goa\\
		403 726, Goa, India}
\email{basilpaul9192@gmail.com}

\address[Pradeep B.]{Department of Mathematics\\
		BITS Pilani K K Birla Goa Campus\\
		Zuarinagar, South Goa\\
		403 726, Goa, India}
\email{pradeepb@goa.bits-pilani.ac.in}

\keywords{Eigenfunctions, Laplacian, Fourier Multipliers, Spherical Fourier Transform, Spherical Functions, Tempered Distributions}
\subjclass[2020]{Primary: 42B10; Secondary: 46E10.}

\begin{document}


\begin{abstract}
In 1993, Robert Strichartz established a characterization for bounded eigenfunctions of the Laplacian on $\mathbb{R}^d$. Let $\left\{f_k \right\}_{k\in \Z}$ be a doubly infinite sequence of functions on $\R^d$ satisfying
$\Delta f_k= f_{k+1}$ for all $k \in \Z$. If $\left\{f_k \right\}$'s are uniformly bounded, then Strichartz proved that $\Delta f_0= f_0$, thus generalizing a classical result of Roe on the real line. Recognizing that many physically significant eigenfunctions exhibit unbounded behavior, Howard and Reese extended this result to include functions of polynomial growth. Building upon a refined functional-analytic framework, we recently established a broader extension of Strichartz's theorem encompassing eigenfunctions of exponential growth. In the present article, we further investigate the spectral geometry of the Laplacian by replacing the differential operator with a broader class of Fourier multipliers. Specifically, we focus on radial convolution operators, including the spherical average, the ball average, and the heat operator. The central problem addressed is as follows: For a fixed multiplier $\Theta$, we consider a doubly infinite sequence of exponentially growing functions $\{f_k\}_{k \in \mathbb{Z}}$ satisfying the recurrence relation $\Theta f_k = A f_{k+1}$ for a complex constant $A$. We demonstrate that under specific spectral conditions, the functions $f_k$ correspond precisely to the eigenfunctions of the Laplacian $\Delta$ on $\mathbb{R}^d$. This result provides a unified approach to characterization theorems, linking the growth rate of eigenfunctions to the symbol of the associated multiplier.
\end{abstract}
	\maketitle
	\section{Introduction}\label{section1}
	
	The problem of characterizing functions based on the behavior of their derivatives or their action under specific operators is a classical theme in harmonic analysis. It stems from the observation that constraints on the growth of a function and its derivatives often force the function to lie within a specific spectral subspace. A foundational result in this direction concerns the real line: any function whose derivatives and anti-derivatives are uniformly bounded must be a linear combination of $\sin x$ and $\cos x$. Building on this observation, John Roe \cite{Roe} established a precise characterization for sine functions in terms of the uniform boundedness of their iterated derivatives.
	
	\begin{theorem}[Roe] \label{th:roe}
		Let $\left\{f_{k}\right\}_{k\in \Z}$ be a doubly infinite sequence of real-valued functions of a real variable with
		\[f_{k+1}(x)=\frac{d}{d x} f_{k}(x).\]
		Suppose further that there exists an $M>0$ such that $|f_{k}(x)| \leq M$  for all $ k \in \Z$ and $x \in \R.$ Then, $f_{0}(x)=a \sin (x+\phi)$ for some real constants $a$ and $\phi$. 
	\end{theorem}
	
	  The transition from one-dimensional characterizations to $\mathbb{R}^d$ requires replacing the derivative $\frac{d}{dx}$ with an appropriate translation-invariant operator, most naturally the Laplacian $\Delta = - \sum_{j=1}^{d} \frac{\partial^2}{\partial x^2_j}$. 

In 1993, Robert Strichartz \cite{Str} extended Roe's theorem to $d$-dimensions. Strichartz's work was motivated by sampling theory and the study of harmonic functions, specifically looking for conditions under which a function is essentially an eigenfunction of the Laplacian.
	
	\begin{theorem}[Strichartz]\label{th:str}
		Let $\left\{f_k \right\}_{k\in \Z}$ be a doubly infinite sequence of functions on $\R^d$ satisfying
		$\Delta f_k=\alpha f_{k+1}$ for some $\alpha$>0, for all $k \in \Z$. If $ \|f_k\|_{L^{\infty}(\R^d)} \leq C$ for all $ k \in \Z$, for some $C>0$ ,  then $\Delta f_0=\alpha f_0$.
	\end{theorem}
	
	Strichartz further remarked that this result holds for $L^p$ norms provided $p> \frac{2d}{d-1}$. However, a significant limitation of these classical results was the restriction to bounded functions. In many contexts involving heat diffusion or wave propagation, solutions naturally exhibit growth. Addressing this, Howard and Reese broadened the scope to account for unbounded eigenfunctions with polynomial growth, bridging the gap between $L^\infty$ spectral theory and the theory of tempered distributions.
	
		\begin{theorem}[Howard and Reese]\label{th:howard2}
		Let $a \geq 0$ and let $\{ f_k\}_{k\in \Z}$ be a sequence of complex-valued functions on $\mathbb{R}^d$ that satisfy
		$$ \Delta f_k=f_{k+1} $$
		and
		\begin{equation}\label{eq:pg} \left|f_k(x)\right| \leq M_k(1+|x|)^a \end{equation} for all $k \in \Z$,
		where the constants $M_k$ have sublinear growth:
		$$ \lim _{k \rightarrow \infty} \frac{M_k}{k}=\lim _{k \rightarrow \infty} \frac{M_{-k}}{k}=0.
		$$
		Then $\Delta f_0=f_0$.
	\end{theorem}
	
	While polynomial growth covers tempered distributions, it fails to capture the spectral geometry of functions with exponential growth, which are central to the study of analytic functionals and Paley-Wiener type spaces. In a recent work \cite{BP}, we generalized Strichartz's theorem to this setting. For a fixed $a>0$, consider the space of complex valued functions on $\R^d$ having exponential growth given by 
	$$X_a = \{f : \R^d \to \C |~ |f(x)| \leq Me^{a|x|}\}.$$ We identified that the point spectrum of Laplacian on $X_a$ is precisely $\Lambda(\Omega_{a})$, where $\Lambda$ denotes the map $\la \mapsto \la^2$ and $\Omega_{a}$ denotes the complex strip $$\{\la \in \C |   |\im(\la)| \leq a\}.$$  On parametrizing the above set, we obtained the spectral geometry of $\D$ on $X_a$  which is a region bounded by a right handed parabola, with its  boundary denoted by $\partial \Lambda(\Omega_{a})$. 
We formulated a characterization for the eigenfunctions of Laplacian having exponential growth corresponding to the complex eigenvalue $\la_{0}^2$ with $\im(\la_0) \neq 0$ as follows.  
	\begin{theorem}[\cite{BP}]\label{Theorem 0.1}
		For any $\la_0 \in \C$ with $\im(\la_0) \neq 0$,  let $N(\la_0)$ denotes the outward normal drawn to the parabola $\partial \Lambda(\Omega_{|\im(\la_0)|})$ at the point $\la_{0}^2$.  Let $z_0 \in N(\la_0)$ and $\{f_k\}_{k \in \mathbb{Z}}$ be a doubly infinite sequence of functions on $\R^d$ satisfying for all $k\in \Z$,
		\begin{enumerate}
			\item $(\D- z_0 I)f_k = Af_{k+1}$ for some non-zero $A \in \C $ and
			\item  $ |f_k(x)| \leq M e^{|\im(\la_0)| |x|} $ for a constant $M>0$ and for every $x \in \R^d$. 
		\end{enumerate} Then the following assertions hold:
		\begin{enumerate}[(a)]
			\item If $|A| = |{\la_0}^2 - z_0|$, then $\D f_0 = {\la_0}^2 f_0$.
			\item If $|A| < |{\la_0}^2 - z_0|$, then $f_k = 0,$ for all $k \in \Z$ and
			\item There are solutions satisfying conditions (1) and (2) which are not eigenfunctions of $\D$ when $|A| >  |{\la_0}^2 - z_0|$.
		\end{enumerate}
	\end{theorem}
However, the Laplacian is but one generator of translation-invariant operators. The history of harmonic analysis, particularly the works of Delsarte on mean-periodic functions and Helgason on homogeneous spaces, suggests that spectral properties of the Laplacian are intimately linked to "functions of the Laplacian," such as spherical means and heat operators. This relationship relies on the Heuristic Principle: an equation involving the Laplacian implies an analogous formulation involving suitable functions of the Laplacian. (See [\cite{mvp}]).

For instance, the classical mean-value property states that a function is harmonic ($\Delta f = 0$) if and only if it is invariant under spherical averaging. This suggests that Strichartz-type characterizations should not be limited to differential operators but should extend to general Fourier multipliers.

The present work generalizes the functional-analytic framework developed in \cite{BP} to obtain characterizations via Fourier multiplier operators. We specifically examine the spherical mean operator $M_t$, the ball mean operator $B_t$, and the heat semigroup. 

%
%

As a representative result, we state below a characterization theorem for the spherical mean operator $M_t$ (see Subsection~\ref{ss:multipliers} for its definition). Throughout, $\phi_{\lambda}(x)$ denotes the Euclidean spherical function defined in \eqref{Sp.funct}, and $e_1=(1,0,\ldots,0)$.

\begin{theorem}\label{theorem1.5}
For a fixed $t>0$ and any $a>0$, let $\{f_k\}_{k\in\mathbb{Z}}$ be a doubly infinite sequence of functions on $\mathbb{R}^d$ satisfying, for all $k\in\mathbb{Z}$,
\begin{enumerate}
\item $M_t f_k = A f_{k+1}$ for some $A\in\mathbb{C}$, and
\item $|f_k(x)| \leq M e^{a|x|}$ for some constant $M>0$ and for all $x\in\mathbb{R}^d$.
\end{enumerate}
Then the following assertions hold:
\begin{enumerate}[(a)]
\item If $|A|=\phi_{ia}(t e_1)$, then $\Delta f_0 = -a^2 f_0$.
\item If $|A|>\phi_{ia}(t e_1)$, then $f_0=0$.
\item If $|A|<\phi_{ia}(t e_1)$, then $f_0$ need not be an eigenfunction of $\Delta$.
\end{enumerate}
\end{theorem}
Theorem~\ref{theorem1.5} may be viewed as a Strichartz-type characterization in the setting of exponential growth, wherein the role of the Laplacian is replaced by the spherical mean operator, thereby extending the classical correspondence between growth conditions and spectral parameters to a broader class of Fourier multipliers.
		
	Analogous results may be obtained for the ball average operator $B_t$ and the heat semigroup $e^{-t\Delta}$, with appropriate modifications of the associated symbols; see Corollaries~\ref{corollary3.5} and~\ref{corollary3.7}. More generally, we aim to establish a Strichartz-type characterization for eigenfunctions of arbitrary Fourier multipliers with real-valued symbols acting on the Schwartz space $\mathcal{S}(\mathbb{R}^d)$.
	 \begin{theorem}\label{theorem1.6}
		Let $\Theta: \mathcal{S}(\R^d) \rightarrow \mathcal{S}(\R^d)$ be a multiplier with real valued symbol $m(\xi) \in C^\infty (\R^d)$. Let $\left\{f_k\right\}_{k\in \Z}$ be a doubly infinite sequence of  functions on $\R^d$. Suppose that for all $ k \in \Z$, $\Theta f_k=A f_{k+1}$ for a non-zero constant $A \in \mathbb{C}$ and $\|f_{k}\|_{p} \leq M $ for $k\in \Z$ and $1\leq p \leq \infty$ and a constant $M >0$. Let $m(\R^d)=\left\{m(\xi)| \xi \in \R^d\right\}$.\\
		\begin{enumerate}[(a)]
			\item  If $|A| \in m(\R^d)$, but $-|A| \notin m(\R^d)$, then $ \Theta f_0=|A| f_0$.
			\item If $-|A| \in m(\R^d)$, but $|A| \notin m(\R^d)$, then $ \Theta f_0=-|A| f_0$.
			\item   If both $|A|$ and $-|A| \in m(\R^d)$, then $f_0$ can be uniquely written as $f_0=f_{+}+f_{-}$, where $f_{+}, f_{-} \in L^p(\R^d)$ satisfying $\Theta f_{+}=|A| f_{+}$ and $ \Theta f_{-}=-|A| f_{-}$.
			\item 	If neither $|A|$ nor $-|A| \in \R^d$, then $f_0=0$.
		\end{enumerate}
	\end{theorem}
	Having clarified the objective, we proceed to formulate a general version of the preceding results so as to place them in a wider analytical framework. The paper is organized as follows. In Section~\ref{section2}, we introduce the basic definitions and set up the functional-analytic framework required for our analysis. In particular, we define a Schwartz-type test function space $S^a(\mathbb{R}^d)$ for $a>0$, together with its dual $S^a(\mathbb{R}^d)'$, the space of exponential tempered distributions of type $a$, whose construction is adapted from the classical theory. The space $S^a(\mathbb{R}^d)$ is chosen so as to consist of smooth functions on $\mathbb{R}^d$ exhibiting exponential decay, and consequently its dual necessarily contains distributions with exponential growth.

We then introduce Fourier multiplier operators on $S^a(\mathbb{R}^d)$ and conclude Section~\ref{section2} with the proof of a fundamental result, referred to as the One Radius Theorem, which reduces the characterization of eigenfunctions of these multipliers to that of eigenfunctions of the Laplacian under suitable conditions.

Section~\ref{section3} is devoted to the proof of the main result of this article, namely an extension of the Strichartz characterization for exponential type $a$ tempered eigendistributions of the Laplacian, obtained via the class of multiplier operators considered herein. Our proof for the main theorems closely follows the methodologies adopted by the authors M. Naik and R. Sarkar in \cite{muna} for characterizing eigenfunctions of the Laplace-Beltrami operator on Riemannian symmetric spaces of non-compact type with real rank one. Finally, in Section~\ref{section4}, we study Fourier multipliers with real-valued and complex-valued symbols on the Schwartz space $\mathcal{S}(\mathbb{R}^d)$, and discuss several special cases as corollaries.
	
	\section{Preliminaries}\label{section2}
	
	\subsection{The Euclidean spherical function and its properties}\label{subsection2.1} In this subsection, we shall study in detail the properties of two important class of functions, first one being the Euclidean spherical function denoted by $\phi_\la$ and the next class denoted by $\psi_{\la}$.
\begin{definition}
		For $\la \in \C$, the Euclidean spherical function $\phi_{\la}$ is defined by 
	
	\begin{equation}\label{Sp.funct}	\phi_\la(x) = \int_{S^{d-1}} e^{i\la x \cdot \om}d\si(\om), ~x\in \R^d   \end{equation}
	where $d\si$ is the normalized surface measure on the unit sphere $S^{d-1}$. 
\end{definition}

Following are some of the important known properties of Euclidean spherical functions. 
	\begin{enumerate}
		\item For $\la \in \R$, $\phi_{\la} \in L^p(\R^d)$ for $ p > \frac{2d}{d-1}$.
		\item For $\la \in \C$, $\phi_\la$ have exponential growth: $|\phi_\la(x)|\leq C_\la |x|^{-\frac{d-1}{2}}e^{|\im(\la)| | x|}.$
		\item $\ds \phi_{\la}(x) = c_d \frac{J_{\frac{d}{2}-1}(\la |x|)}{(\la |x|)^{\frac{d}{2}-1}}$, where $ J_\ap(t)$ is the Bessel function of type $\ap$.
		\item $|\phi_{\la}(x)| \leq \phi_{i \im(\la)}(x)$
		\item For $x \in \R^d, \ds \left|P \Big(\frac{\partial}{\partial \la}\Big)\phi_{\la}(x) \right| \leq C (1+|x|)^r \phi_{i \im(\la)}(x),$ where $P$ is a polynomial of degree $r$.
	\end{enumerate}
Recall that $\Omega_a=\{\la \in \C |   |\im(\la)| \leq a\}$ and $e_1=(1, 0, \ldots, 0)$.  We state and prove certain properties of $\phi_\la$ which are useful in proving the main results.
		\begin{proposition} \label{Prop2.2}
			For a fixed $t>0$ and for any $a>0$, consider any $\la \in \Omega_{a}$ with $\la \neq \pm ia$. Then $$|\phi_{\la}(te_1)| < \phi_{ia}(te_1).$$
		\end{proposition}
		\begin{proof}
			We have already seen that $|\phi_{\la}(te_1)| \leq \phi_{i\im(\la)}(te_1)$, for any $\la \in \C$.  Now the map $\la \mapsto \phi_{\la}(te_1)$ being analytic, the well known maximum modulus principle would yield us $|\phi_{\la}(te_1)| < \phi_{ia}(te_1)$ for $\la \in \Omega_{a}^o$. Therefore it is enough to show that for $\la =r+i a, r\in \R$, $|\phi_{\la}(te_1)| < \phi_{ia}(te_1)$. On the contrary, let us assume $|\phi_{r+ia}(te_1)|= \phi_{ia}(te_1)$ for some $r \in \R$. Then for some $b\in \R$, we have $\phi_{ia}(te_1) = e^{ib} \phi_{r+ia}(te_1)$.
			This implies that $$\ds \int_{S^{d-1}} e^{at e_1\cdot \omega}[1- e^{i(b-rte_1\cdot \omega)}] d\si (\omega) = 0.$$ Hence on taking the real part, we obtain $$\ds \int_{S^{d-1}} e^{ate_1\cdot \omega} [1 - \cos(b-rte_1\cdot \omega)] d\si(\omega) = 0.$$  Since $1-\cos(b-rte_1\cdot \omega)$ is a non-negative continuous function, we must have $[1-\cos(b-rte_1\cdot \omega)]=0$, for every $\omega \in S^{d-1}$. Thus the mapping $\omega \mapsto rte_1\cdot \omega$ must be constant which is not true. Hence $|\phi_{\la}(te_1)| < \phi_{ia}(te_1)$ for all $\la \in  \Omega_{a}$.
			
		\end{proof}
		
		\begin{proposition}\label{prop 2.3}
			For a fixed $x\in \R^d-\{0\}$ and for any $a>0, ~|\phi_{\la}(x)| \to 0$ uniformly on $\Omega_{a}$, as $|\la| \to \infty$.
		\end{proposition}
		\begin{proof} For $\la=\ap+i\bt$, consider
			\begin{align*}
				\phi_\la(x) & = \int_{S^{d-1}} e^{i\la x\cdot \om}d\si(\om) = C_d \int_{-1}^{1} e^{i(\ap + i\bt)rt } (1-t^2)^\frac{d-3}{2} dt = \int_{-1}^{1} g_r(\bt, t)  e^{i \ap rt} dt 
			\end{align*} where $g_r (\bt, t) = C_d (1-t^2)^\frac{d-3}{2} e^{-\bt rt}$.
				On a suitable change of variable, we obtain $$ \phi_\la(x) = \int_{-1-\frac{\pi}{\ap r}}^{1- \frac{\pi}{\ap r}} g_r (\bt, t+ \frac{\pi}{\ap r}) e^{i \ap rt} e^{i \pi} dt. $$
				Therefore \begin{align*}
					2 |\phi_{\la}(x)| & = \bigg| \int_{-1}^{1} g_r(\bt, t)  e^{i \ap rt} dt - \int_{-1-\frac{\pi}{\ap r}}^{1- \frac{\pi}{\ap r}}  g_r(\bt, t+ \frac{\pi}{\ap r}) e^{i \ap rt} dt \bigg| \\& \leq \int_{-1}^{1- \frac{\pi}{\ap r}} |g_r(\bt, t) - g_r(\bt, t+ \frac{\pi}{\ap r})| dt + \int_{1-\frac{\pi}{\ap r}}^{1} |g_r(\bt, t)| dt + \int_{-1}^{-1+\frac{\pi}{\ap r}} |g_r(\bt, t) | dt    
				\end{align*} which goes to zero as $\ap\to \infty$ uniformly on $ |\bt| \leq a$.
		\end{proof}
		
		In what follows let  \begin{equation}\label{eq:derofphi} \ds \phi_{\mu,k}(x) := \frac{\partial^k}{\partial \la^k} \phi_{\la}(x) \Big|_{\la = \mu},\end{equation}  for $\mu \in \C$ and $k \in \mathbb{N}$.
		\begin{proposition}\label{prop 2.4}
			For a fixed $t>0$, we have $\phi_{0,2}(te_1) \neq 0$ and $\phi_{\la,1}(te_1) \neq 0$ for non-zero $\la \in i\R$.
		\end{proposition}
		\begin{proof}
			As we know $\phi_{\la}(te_1) =\ds  \int_{S^{d-1}} e^{i\la te_1\cdot \omega} d\si(\omega)$. Thus for $\la \in i\R, \phi_\la(te_1) > 0$. Now consider $$\phi_{\la,2}(te_1)\Bigr \rvert_{\la= i\bt} = \frac{\partial^2}{\partial \la^2} \phi_\la (te_1) \Bigr \rvert_{\la= i\bt} = - \ds \int_{S^{d-1}} t^2(e_1\cdot\omega)^2 e^{-\bt te_1\cdot\omega} d\si (\omega) < 0.$$ So in particular, $\phi_{0,2}(te_1) \neq 0.$

			It is already known that for a fixed $t>0$, the map $\la \mapsto \phi_\la(te_1)$ is a non-constant entire function. By the maximum modulus principle, we obtain $\phi_{iy_1}(te_1) < \phi_{iy_2}(te_1)$ for $0 \leq y_1 < y_2$. That is, the function $f: y \mapsto \phi_{iy}(te_1) $ is strictly increasing. And the second derivative of $f$ with respect to $y$ being strictly negative as seen above, $f$ has nonzero derivative at any $y>0$, because otherwise it will have a local maximum. Hence $\phi_{\la,1}(te_1) \neq 0$ for any nonzero $\la \in i\R$.
		\end{proof}
		Here we shall introduce another class of functions denoted by $\psi_{\la}$, that will be used recurrently.
		\begin{definition} \label{defn 2.5}
			For $\la \in \C$, we define $\psi_{\la}$ as  \begin{equation}\label{eq:fcub} \psi_{\la}(\xi) = \frac{1}{|B(0,1)|} \int_{B(0,1)} e^{-i \la \xi \cdot y} dy,~ \xi \in \R^d \end{equation} where $B(0,1)$ denotes the ball centered at origin with radius 1 in $\R^d$.
		\end{definition}
		
		\begin{proposition} 	
			For a fixed $x\in \R^d-\{0\}$ and for any $a>0$, consider any $\la \in \Omega_{a}$ with $\la \neq \pm ia$. Then $|\psi_{\la}(te_1)| < \psi_{ia}(te_1)$.
			\end{proposition}
			\begin{proof}
				Proof of the above proposition follows from a step by step adaptation of the arguments used in the proof of Proposition \ref{Prop2.2}.
			\end{proof}
				\begin{proposition}\label{prop 2.7}
				For a fixed $x\in \R^d-\{0\}$ and for any $a>0, |\psi_{\la}(x)| \to 0$ uniformly on $\Omega_{a}$, as $|\la| \to \infty$.
			\end{proposition}
			\begin{proof}
				Note that, on applying a suitable change of variable on the definition of $\psi_\la$ defined as in \eqref{eq:fcub}, we obtain $$ \psi_{\la}(t) = \frac{1}{|B(0,t)|} \int_{0}^{1} \phi_{\la}(rt \omega) r^{d-1} dr .$$ Now by using Proposition \ref{prop 2.3} we obtain the required proof. 
			\end{proof}
		\begin{proposition}\label{prop 2.8}
		For a fixed $t>0$, we have $\psi_{0,2}(te_1) \neq 0$ and $\psi_{\la,1}(te_1) \neq 0$ for non-zero $\la \in i\R$.
	\end{proposition}
	\begin{proof}
		Proof follows by the similar line of argument as in the proof of Proposition \ref{prop 2.4}.
	\end{proof}
		\subsection{Schwartz type space $S^a(\R^d)$ and $H_e(\Omega_{a})$}\label{subsection2.2}:

		We shall introduce a \textit{Schwartz type space} as follows. For a fixed real number $a \geq 0$, let $S^a(\R^d)$ be the space of all $C^{\infty}$ functions on $\R^d$ such that 
		
		$$ \gm _{m,\ap}(f) = \sup_{x \in \R^d} e^{a|x|} (1+|x|)^m |D^\ap f(x)| < \infty$$
		for every non-negative integer $m$ and for all multi-index $\ap=(\ap_1, \ldots, \ap_d)\in \mathbb{N}_0^d$, where $$\displaystyle D^\ap f = \frac{\partial^{|\ap|}f}{{\partial x_1}^{\ap_1}...\partial {x_d}^{\ap_d}}.$$

We note that for $ a=0$, the space $\sa(\R^d)$ coincides with the standard Schwartz space $\mathcal{S}(\R^d)$. Let $\sar(\R^d)$ denote the space of all radial functions in $S^a(\R^d)$. For $a>0$, we recall that $\Omega_{a} = \{\la \in \C | |\im(\la)|\leq a\}$. Let $H_e(\Omega_{a})$ be the set of all even holomorphic functions on $\Omega_{a}^o$ which are continuous on $\Omega_{a}$ and satisfying for all non-negative integers $m$ and for all multi-indices $\ap$ 
		$$\mu_{m,\ap}(\phi) = \sup_{\la \in \Omega_{a}} (1+|\la|)^m \Big|\Big(\frac{d}{d\la}\Big)^\ap \phi (\la)\Big|< \infty. $$
		
		It can be seen that both spaces $S^a(\R^d)$ and $H_e(\Omega_{a})$ are Fr\'echet spaces with respect to the topology induced by the family of seminorms given by $\gm _{m,\ap}$ and $\mu_{m,\ap}$ respectively.
		
		Now for any function $f \in \sa(\R^d)$, we define the spherical Fourier transform of $f$ as
		$$ \mathcal{H}f(\la)= \int_{\R^d} f(x) \phi_{\la}(x)dx,~\la \in \Omega_a. $$
		For $f \in \sar(\R^d) $, it is easy to see  that $ \mathcal{H}f$ is an even holomorphic function on $\Omega_{a}$. In fact, we have that $\mathcal{H}: \sar(\R^d)\to H_e(\Omega_{a})$ is a topological isomorphism.  For more details, see   \cite[Theorem 2.1]{natan}. Using classical Fourier inversion formula, we can show that $\mathcal{H}^{-1}: H_e(\Omega_{a}) \to \sar(\R^d)$ with 
\begin{equation}\label{eq2}
	\mathcal{H}^{-1}\phi(x)=\int_{0}^\infty \phi(\la)\phi_\la(x) \la^{d-1}d\la
\end{equation}	 for all $ \phi \in H_e(\Omega_{a})$.\\
		At this point we intend to provide a vital result which could be useful while proving the main theorems. We omit its proof as it only requires an adaptation of the same strategy as devised in the proof of Theorem 3.2 provided in  \cite{weit}.
		\begin{proposition}\label{prop2.9}
			For $0<b<a$, let $ \set{g_0,g_1,...,g_r}$ be a finite collection of functions in $\sar(\R^d)$ with $g_0$ satisfying $$\ds \lim_{|t| \to \infty} \sup e^{\frac{-\pi |t|}{2a}} \log |\mathcal{H}{g_0}(t)|=0.$$ Let $A$ be the set of common zeroes of the functions $\mathcal{H}{g_0}, \mathcal{H}{g_1},... \mathcal{H}{g_r}$ in $\Omega_{a}$ and $n_\la $ denote the minimal order of zero of the functions $\mathcal{H}{g_i}$ at the point $\la$. Assume that $A$ is finite and $A \subseteq  \Omega_{b}$. Then the ideal generated by 
			$ \set{g_0,g_1,...,g_r}$ is dense in the closed subspace of $\sbr(\R^d)$ of functions whose spherical Fourier transform vanish at all $\la \in A$ with order greater than or equal to $ n_{\la}.$
		\end{proposition}
		\subsubsection{Extension of spherical Fourier transform to corresponding dual spaces $S^a (\R^d)'$ and $H_e(\Omega_{a})'$}\label{subsection2.1.2}

		In what follows $H_e(\Omega_{a})'$ denotes the dual of $H_e(\Omega_{a})$ and for  any $S \in H_e(\Omega_{a})'$, for any suitable even holomorphic function $\psi$ on $\Omega_a$,  we define $ \psi S$ as a dual element by the equation $$\langle \psi S, \phi \rangle = \langle S,  \psi \phi \rangle, $$ for all $\phi \in H_e(\Omega_{a})$,  where $\langle . , . \rangle$  denotes the dual bracket. 
		
		Now we name $\sa(\R^d)'$, the dual space of $\sa(\R^d)$ as the space of  all exponential tempered distributions of type $a$ on $\R^d$. For any $\psi \in \sa(\R^d)$, we define the action of the operator $\D$ on $T \in  S^a(\R^d)'$ as follows:  $\langle \D T, \psi \rangle = \langle T,  \D \psi \rangle $.
		For an exponential type tempered distribution $T$, its spherical Fourier transform $\mathcal{H}T$ is defined as a linear functional on $H_{e}(\Omega_{a})$ by the following rule:
		$$ \langle \mathcal{H}T, \phi \rangle = \langle T,\mathcal{H}^{-1}\phi \rangle $$
		where $\phi \in H_{e}(\Omega_{a})$ and $ \mathcal{H}^{-1}\phi $ as given in (\ref{eq2}). 
		
		\subsection{Radialization operator}
		For a suitable function $f$ on $\R^d$, its radialization $Rf$ is defined as $$ Rf(x) = \int_{S^{d-1}} f(|x| \omega) d\sigma(\omega).$$ The following properties of the radialization operator are essentially needed.
		\begin{itemize}
				\item $$\int_{\R^d} Rf(x)g(x) dx  = \int_{\R^d} f(x) Rg(x) dx , $$where $f, g \in S^a(\R^d)$
			\item $ R(\D f)= \D (Rf)$

			\noindent 
		\end{itemize}
		The above definition of radialization can be extended to $T \in \sa(\R^d)' $ as follows.
		\begin{definition}
				For any $T \in \sa(R^d)'$, the radialization of $T$ is defined by the rule $\langle RT, f \rangle = \langle T, Rf \rangle$ for all $f \in \sa(\R^d)$. Thus $T \in \sa(\R^d)'$ is said to be radial if it satisfies $RT = T$.
				
			\end{definition}

		\subsection{Translation operator}
		 Given any $y \in \R^d$ and a function $f: \R^d \to \C,$ we define the translation of $f$ as $(\ell_yf)(x)= f(x-y)$.  For $T \in S^a(\R^d)'$ and $y \in \R^d$ , we define the translation $\ell_yT$ via $$\langle \ell_yT, f \rangle= \langle T, \ell_{-y} f \rangle, $$ for every $f \in S^a(\R^d)$.
		 
		 \subsection{Convolution}
		For $T \in \sa(\R^d)'$ and for $f \in \sa(\R^d)$, we define $$ T* f(x) = T(\ell_x f).$$ For $f_1, f_2 \in \sa(\R^d)$ and $T \in \sa(\R^d)'$, it can be verified that $$ \langle T, f_1*f_2 \rangle = \langle T * f_1, f_2 \rangle.$$

		\subsection{Multiplier operators}\label{ss:multipliers}
                 For $f\in \sa(\R^d)$, $\hat{f}$ denotes the Fourier transform of $f$ and defined by
$$\hat{f}(\xi)=\int_{\R^d}f(x) e^{-i x\cdot \xi}dx.$$		
		\begin{definition}
		 Let $m$ be a $C^{\infty}$ function defined on $\R^d$. 	For a fixed $a>0$, a continuous linear operator $\Theta : S^a(\R^d) \to S^a(\R^d)$ satisfying the relation $$ \widehat{(\Theta f)}(\xi) = m(\xi) \hat{f}(\xi)$$ for all $f \in S^a(\R^d),  ~\xi \in \R^d$ is called a (Fourier) multiplier on $S^a(\R^d)$ with symbol $m(\xi)$.		\end{definition}
			Following are the multiplier operators which we will be using to characterize eigenfunctions of Laplacian in this work.
			 \begin{enumerate}[(1)]
			 	
			\item	{\bf Spherical average.} For $f \in L^1_{loc}(\R^d)$ and $t>0$, the spherical average of $f$ is denoted by $M_tf(x)$ and is defined by $$M_tf(x) = \ds \int_{S^{d-1}} f(x-t\omega) d\si(\omega).$$

			\item {\bf Volume average.}  Let $B(x,t)$ be the ball of radius $t>0$ centered at $x \in \R^d$ and $|B(x,t)|$ be its volume. The volume average of a function $f \in  L^1_{loc}(\R^d)$, denoted by $B_t f(x)$ and defined by $$ B_t f(x) = \frac{1}{|B(0,t)|} (f* \chi_{B(0,t)})(x).$$
\item {\bf Heat operator.} The heat operator $e^{-t\D}$ on $\R^d$ is given by $$ e^{-t\D}f(x) = (f*h_t)(x)\text{ where}~  h_t(x) = \frac{1}{(4\pi t)^\frac{d}{2}} e^{\frac{-|x|^2}{4t}},\text{for}~~ t>0.$$ In fact $h_t$ is called the heat kernel on $R^d$ and it is the fundamental solution of the heat equation $ -\D u(x,t)= \frac{\partial}{\partial t} u(x,t). $ 
	\end{enumerate}
	The following proposition establishes the fact that the above-mentioned operators are indeed Fourier multiplier operators on the space $S^a(\R^d)$.
	\begin{proposition}
	Let $f \in S^a(\R^d)$. Then for any $t>0,$
	 \begin{enumerate}
		\item $M_tf ~\text{and}~B_tf$ define  multiplier operators on  $S^a(\R^d)$ with symbols $\phi_t(\xi)$ and $\psi_{t}(\xi)$ respectively.
		
		\item $e^{-t\D}$ defines a multiplier operator on  $S^a(\R^d)$ with symbol $e^{-t |\xi|^2}$.
	\end{enumerate}

	\end{proposition}
		
	\begin{proof} Initially we shall establish the continuity of the operator $M_t$ on $S^a(\R^d)$.
		It is easy to see that $D^\ap M_tf(x) =  \ds \int_{S^{d-1}} D^\ap f(x-t\omega) d\si(\omega) = M_t(D^\ap f)(x)$. Then
		\begin{align*}
			\gamma_{m, \ap}(M_tf) &= \sup_{x \in \R^d}  \ds \int_{S^{d-1}} (1+|x|)^m e^{a|x|}D^\ap f(x-t\omega)  d\si(\omega) \\ & = \sup_{x \in \R^d}  \ds \int_{S^{d-1}} \frac{(1+|x|)^m}{(1+|x-t\omega|)^m} \frac{e^{a|x|}}{e^{a|x-t\omega|}} (1+|x-t\omega|)^m e^{a|x-t\omega|}D^\ap f(x-t\omega)  d\si(\omega)\\ &\leq  \sup_{x \in \R^d}  \ds \int_{S^{d-1}} \frac{(1+|t|)^m e^{a|x|}}{e^{a|x|} e^{-a|t|}} \gamma_{m, \ap}(f) d\si(\omega) \\ & = (1+|t|)^m e^{a|t|} \gamma_{m, \ap}(f).
		\end{align*} 
		 Thus $f \mapsto M_tf$ is continuous on $S^a(\R^d)$. Note that $M_tf(x) = (f * \si_t)(x)$, where 
		 $\si_t$ denotes the surface measure on the sphere of radius $t$ with center origin.	As $\widehat{M_tf}(\xi) = \hat{f}(\xi) \phi_t(\xi)$, we see that $M_t$ being a multiplier operator with symbol $\phi_{t}$.

Now continuity of the operator $B_t$ and $e^{-t\D}$ can be proved using the similar arguments used above. Then $\widehat{B_tf}(\xi) = \hat{f}(\xi) \psi_t(\xi)$ implies that $B_t$ being a multiplier operator with symbol $\psi_{t}.$, whereas $\widehat{e^{-t\D}f}(\xi) = \hat{f}(\xi) e^{-t|\xi|^2}$ implies that $e^{-t\D}$ is a multiplier with symbol $ e^{-t|\xi|^2}$.
	\end{proof}

\begin{remark}
The multipliers are naturally extended to the dual space in the following way: If $\Theta$ is a multiplier with symbol $m(\xi)$, then for any $T\in S^a(\R^d)'$, $\Theta T$ is defined by $\langle \Theta T, f\rangle=\langle T, \Theta^* f\rangle$ for all $f\in S^a(\R^d)$, where $\Theta^*$ is a multiplier  with the symbol $m(-\xi).$ 
\end{remark}	

\begin{remark}
	It can be verified that for all $ \phi \in H_e(\Omega_{a})$ and $T \in \sa(\R^d)'$, we have $ \langle \mathcal{H} M_t T, \phi \rangle = \langle \phi_{\la}(te_1) \mathcal{H}T, \phi \rangle$, i.e., $\mathcal{H}M_t T=\phi_\la(te_1)\mathcal{H}T$.
\end{remark}
	
	\subsection{One Radius Theorem} 
	
	In this subsection we shall establish a result of significant importance as it aids in characterizing the eigendistributions of Laplacian with exponential growth through mean value property by considering a fixed radius alone. Though we are proving this result for the spherical average $M_t$, an analogous version can be proved for ball average $B_t$ also by replacing appropriate symbols.
	\begin{theorem}\label{theorem2.15}
		
	Let $t>0$ be fixed and $a \geq 0$. Suppose $T \in S^a(\R^d)'$ be such that $M_tT= \phi_{ia}(te_1)T$. Then $T$ is an eigendistribution of $\D$ with eigenvalue $-a^2$.
\end{theorem}

\begin{proof}
	It can be seen that for $T \in S^a(\R^d)'$,  we have \begin{equation}
	M_t T = T * \si_t
	\end{equation} where $\si_t$ denotes the surface measure on the sphere of radius $t$ with center origin.	Let $\delta_0$ be the dirac delta distribution. Together with the  hypothesis $M_tT = \phi_{ia}(te_1)T$ implies that $T * \nu = 0,$ where $ \nu = \si_t - \phi_{ia}(te_1) \delta_0$.
	
	Initially we shall consider the case when $T$ is radial and $a > 0$. Let $$\Psi_1(\la) = \mathcal{H}{\nu(\la)}= \mathcal{H}{\si_t} - \phi_{ia}(te_1) \mathcal{H}{\delta_0}= \phi_{\la}(te_1) - \phi_{ia}(te_1)$$ where $\la \in \C. $
	Now we assert that there exists a $ \delta>0$ for which $\Psi_1(\la) $ remains non-zero throughout the strip $\Omega_{a+\delta}$ except precisely at the points $\la= \pm ia.$ It follows from Proposition \ref{Prop2.2} that if $\la \in {\Omega_{a}}$ and $\la \neq \pm ia$, then $\phi_{\la}(te_1) \neq \phi_{ia}(te_1)$. Let $\delta_1 >0$ be fixed. Since $|\phi_{\la}(te_1)| \to 0,$ when $|\la| \to \infty$, we can find an $N>0$ such that $|\phi_{\la}(te_1)| < \frac{\phi_{ia}(te_1)}{2}$, whenever $ \la \in {\Omega_{a+\delta_1}}$ and $|\re(\la)| > N$. The compactness of the set $\{ \la| \la \in {\Omega_{a+\delta_1}}, |\re(\la)| \leq N \}$  together with the analyticity of the map $\la \mapsto \phi_{\la}(te_1)$ implies that the set $Z = \{ \la \in  {\Omega_{a+\delta_1}} | \la \neq \pm ia, |\re(\la)| \leq N, \phi_\la(te_1) = \phi_{ia}(te_1)\}  $ must be necessarily finite. If $Z$ is empty, we set $\delta = \delta _1$, or otherwise we may choose any $\delta < \rho'-a$, where $\rho' = \inf_{\la \in Z} |\im(\la)|$. This ensures that the only zeroes of $\Psi_1$ in ${\Omega_{a+\delta}}$ are $\pm ia$. Moreover Proposition \ref{prop 2.3} asserts that $\phi_{ia,1}(te_1) \neq 0$, implying $\Psi_1$ has simple roots at $\pm ia$.

	Let $ \Psi_2(\la)= e^{-\la^2}\Psi_1(\la), \la \in {\Omega_{a}}$. Then it can be verified that $\Psi_2 \in H_e(\Omega_{a}).$ Let $g \in \sar(\R^d)$ such that $\mathcal{H}{g} = \Psi_2$.  By applying Proposition \ref{prop2.9}, we obtain that $\{g*h | h \in \sar(\R^d)\}$ is dense in the space of all functions in $\sar(\R^d)$ whose spherical Fourier transform vanishes at $\pm ia$.  Since $\langle T, g*h \rangle =0$ for any $h\in \sar(\R^d)$, $ \langle T, \phi \rangle =0$ for every $\phi \in \sar(\R^d)$ with $ \mathcal{H}{\phi}(ia) = \mathcal{H}{\phi}(-ia)=0.$ However, $\phi_{ia}(x)$ is itself a radial tempered distribution of exponential type that annihilates every $\phi \in \sar(\R^d)$ satisfying $\mathcal{H}{\phi(ia)}=0$.
    Consequently by Lemma 3.9 in \cite{rudin}, we have $T=M \phi_{ia}$ for some constant $M$. In particular, $T$ is an eigendistribution of $\D$ with eigenvalue $-a^2$.  
	
	Let us now consider the case when $a=0$. By Theorem \ref{theorem4.7},  we obtain  Supp $ \mathcal{H}T \subseteq \{0\}$. Lemma \ref{lemma4.8} yields that $T$ must be of the form $\sum_{k=0}^{N} a_k \phi_{0,k}$ for some constants $a_0, a_1, \dots a_N $ such that $ a_N \neq 0$. On using the fact that $\phi_{\la}(x)$ being an even function of $\la$, $\phi_{0,k}= 0$ for $ k$ odd, we obtain $T= \sum_{k=0}^{N} a_k \phi_{0,2k}$. Proof for this part will be done once we establish that $N=0$. Note that \begin{align*}
		M_t \phi_{0,2k}(x) & = \frac{\pa^{2k}}{\pa \la^{2k}}\Big|_{\la = 0}(\phi_{\la}*\si_t)(x) =  \frac{\pa^{2k}}{\pa \la^{2k}} \Big|_{\la = 0} (\phi_{\la}(te_1)\phi_{\la}(x)) \\ & = \sum_{i=0}^{2k} \binom{2k}{i} \phi_{0,i}(te_1) \phi_{0,2k-i}(x) 
		= \sum_{i=0}^{k} \binom{2k}{2i} \phi_{0,2i}(te_1) \phi_{0,2(k-i)}(x).
			\end{align*} 
		Assuming $N \geq 1$ and using the hypothesis $M_t T = \phi_0(te_1) T$, comparing the coefficient of $\phi_{0,2(N-1)}$ on both sides we get 
	\[
	N(2N-1) a_N \phi_{0,2}(te_1) = 0.
	\]
	But since $\phi_{0,2}(te_1) \neq 0$ (see Proposition \ref{prop 2.4}), we get $a_N = 0$, which is a contradiction. Hence $T = a_0 \phi_0$ and $\Delta T = 0$. This completes the proof of the theorem for radial distributions.

	Next we shall prove the theorem when $T_k$'s are non-radial. Initially we shall see that given any non-radial exponential type tempered distributions $T$ satisfying the hypothesis of the theorem, the corresponding radial distribution ${R\ell_xT}$  for any $x \in \R^d$, also satisfies the hypothesis, which is evident from the following:
	$$M_t(R \ell_xT ) = R(\ell_x T)* \si_t = R(\ell_x(T*\si_t)) = R(\ell_x(\phi_{ia}(te_1))T) = \phi_{ia}(te_1) R(\ell_xT) .$$

	From the result proved for radial distributions, we conclude that $\D R(\ell_x T) = -a^2 R(\ell_x T)$, for every $x \in \R^d$ which implies $R\ell_x(\D T) = R\ell_x(-a^2 T)$, for every $x \in \R^d$. It is enough to prove that for any $T \in S^a(\R^d)'$, if $R(\ell_xT)=0$ for all $x \in \R^d$, then $T$ equals zero as a distribution. Indeed if $R(\ell_xT)=0$ for all $x \in \R^d$, then $ \langle \ell_x T, h_t \rangle = 0$ for all $t > 0$ where $ h_t$ denotes the heat kernel, which is a radial function given by $\ds h_t(x) = \frac{1}{(4\pi t)^\frac{d}{2}} e^{\frac{-|x|^2}{4t}}$, for $t>0$. That is, $T * h_t \equiv0. $ But $T * h_t \to T$ as $t \to 0$ in the sense of distributions. Therefore $\D T+ a^2T=0$ and hence $\D T= -a^2 T$. 
\end{proof}
		
		Below given propositions are equivalent versions of One Radius Theorem for the ball average $B_t$ and heat operator $e^{-t \D}$, whose proof we omit as it follows the same line of arguments as above.
			\begin{proposition}
			
			Let $t>0$ be fixed and $a\geq0$. Suppose $T \in S^a(\R^d)'$ be such that $B_tT= \psi_{ia}(te_1)T$. Then $T$ is an eigendistribution of $\D$ with eigenvalue $-a^2$.
		\end{proposition}
		\begin{proposition}
			Let $t>0$ be fixed and $a >0$. Let $T \in S^a(\R^d)'$ be such that $ e^{-t \D}T = e^{ta^2}T$. Then $T$ is an eigendistribution of $\D$ with eigen value $-a^2$.
		\end{proposition}
		
		\begin{proposition}
				Let $t>0$ be fixed and $a = 0$. Let $T \in S^a(\R^d)'$ be such that $ e^{-t \D}T = e^{-t\la_{0}^2}T$ for some $\la_0 \in \R$. Then $T$ is an eigendistribution of $\D$ with eigenvalue $\la_{0}^2$.
		\end{proposition}
		
		\begin{proof}  For the sake of brevity, we shall give the proof for the case when $T$ is radial, as for the non-radial case the same arguments can be used as in Theorem \ref{theorem2.15}. By Theorem \ref{theorem4.7}, we obtain that $\text{Supp} \mathcal{H}T \subseteq \{ \la \in \R \mid |e^{-t\la^2}| = e^{-t\la_{0}^2}\} = \{\pm \la_{0}\}$. Now Lemma \ref{lemma4.8} yields that $T$ must be of the form $P_1(\pa_{\la}) \phi_{\la} |_{\la= \la_{0}} + P_2(\pa_{\la}) \phi_{\la} |_{\la= -\la_{0}} $. Since $\phi_{\la}$ being even as a function of $\la$, we shall further simplify  the last expression to obtain $T = \sum_{k=0}^{N} a_k \phi_{\la_{0},k}(x)$ where $a_0, a_1, \ldots, a_N$ are constants with $a_N \neq 0$. It can be easily verified that $$ e^{-t \D} T = \sum_{k=0}^{N} a_k \frac {\pa^k}{\pa \la^k }(e^{-t\la^2} \phi_{\la})|_{\la =\la_0}.$$ Thus it follows from the hypothesis that 
$$ e^{-t\la_{0}^2} \sum_{k=0}^{N} a_k \phi_{\la_{0},k} = \sum_{k=0}^{N} a_k \sum_{j=0}^{k} \binom{k}{j}  \phi_{\la_{0},j} \frac{d^{k-j}}{d\la^{k-j}} (e^{-t\la^2})|_{\la= \la_0}. $$
		Assuming $N \geq 1$, on comparing the coefficient of $\phi_{\la_0,N-1}$ on both sides we get $ a_{N}N e^{-t\la_{0}^2}(-2t\la_{0}) =0$ which implies that $a_{N} =0$ if $\la_0\ne 0$.  Therefore $T =a_{0} \phi_{\la_{0}}$ and hence $  \D T = \la_{0}^2 T$. Similar arguments work for the case $\la_0=0$ on comparing the coefficients of $\phi_{0, N-2}$ and we get that $N\leq 1$. As a consequence $T$ will be a polynomial in $x$ of degree $1$. Therefore $\D T=0$.
		\end{proof}
\section{Characterization of exponential type tempered eigendistributions}\label{section3}
This section presents our central results- extended version of Strichartz's theorem for spherical averages, ball averages and the heat operator. First we shall present couple of introductory lemmas and later proceed to the proof of our main results Theorem \ref{theorem3.3}, \ref{theorem3.4} and \ref{theorem3.6}.
\begin{lemma}\label{lemma3.1}
	Let $\Theta: S^a(\R^d) \to S^a(\R^d)$ be a multiplier and $\{T_k\}_{k \in \Z^+}$ be an infinite sequence of radial exponential type tempered distributions. Suppose that $\{T_k\}_{k \in \Z^+}  $ satisfies the following conditions: \begin{enumerate}
		\item    $ \Theta T_k= A T_{k+1}$ for a nonzero constant $A \in \C $ and for all $k \in \Z^+$.
		\item $|\langle T_k, \psi \rangle| \leq M \gamma(\psi)$ for a fixed seminorm $\gamma$ of $S^a(\R^d)$ and a constant $M>0$.
	\end{enumerate} If $(\Theta-B)^{N+1} T_0 = 0$ for some $ B \in \C$ with $|B| =|A|$ and $N \in \Na $, then $ \Theta T_0 = BT_0$. 
\end{lemma}
\begin{proof}
 It follows from the hypothesis  $(\Theta-B)^{N+1} T_0 = 0$ that 
	\begin{equation*}
		Span\{T_0, T_1, \ldots\}= Span\{T_0, \Theta T_0,\ldots, \Theta^N T_0\}= Span\{T_0,T_1, \ldots,T_N\}
	\end{equation*}
		We will  prove that $(\Theta-B)T_0 =0$. On the contrary, suppose that $(\Theta- B)T_0 \neq 0$. Let $k_0$ be the largest positive integer such that $(\Theta -B)^{k_0}T_0 \neq 0$. Clearly $k_0 \leq N$. 

Let $T=(\Theta-B)^{k_0-1}T_0$, then $T \in Span \{T_0,T_1, \ldots,T_N\}$. Therefore we assume $T = \sum_{j=0}^N a_jT_j$. Then 
	\begin{equation}\label{eqn8}
		(\Theta-B)^2 T=0\quad \mbox{and} \quad (\Theta-B)T \neq 0.
	\end{equation}
	
	Now, using \eqref{eqn8} and binomial expansion, we obtain \begin{equation*}
		\Theta^k T= ((\Theta-B)+B)^kT
		= B^k T+ k B^{k-1}(\Theta-B)T.
	\end{equation*} for $k\geq 2$. Hence for any $\psi \in \sa(\R^d)$,
	\begin{equation}\label{eqn9}
		|\langle (\Theta-B)T, \psi  \rangle| \leq \frac{1}{k}|A|^{1-k} |\langle \Theta^k T, \psi \rangle | + \frac{1}{k} |A||\langle T, \psi \rangle|.
	\end{equation}
	Consider 
	\begin{align*}
		|\langle \Theta^k T, \psi \rangle| & =\left|\left\langle \Theta^k \sum_{j=0}^N a_jT_j, \psi \right\rangle \right| =\left |\left \langle \sum_{j=0}^N a_jA^kT_{j+k}, \psi \right\rangle \right|\\
		& = |A|^k\left|\left \langle \sum_{j=0}^N a_j T_{j+k}, \psi \right \rangle \right | \leq |A|^k \sum_{j=0}^N \left |a_j\right| \left|\left \langle T_{j+k}, \psi \right\rangle \right| \\
		& \leq M |A|^k \gm(\psi) \sum_{j=0}^N |a_j|.
	\end{align*}
	
	Thus \eqref{eqn9} and above inequality implies that,
	\begin{equation*}
		|\langle (\Theta-B)T, \psi  \rangle| \leq \frac{M}{k} |A| \gm(\psi) \sum_{j=0}^N |a_j| + \frac{1}{k} |A| |\langle T, \psi \rangle|.
	\end{equation*}As the right hand side goes to 0 as $ k \rightarrow \infty$, we get $(\Theta-B)^{k_0}T_0 = 0$, which contradicts the assumption on $k_0$. Thus, we obtain $ (\Theta-B)T_0 =0 $. 
\end{proof}

\begin{lemma} \label{lemma3.2}
	Let $t>0$ and $a>0$ be fixed. For $A \in \C$ with $|A| < \phi_{ia}(te_1),$ there exists infinitely many $\la \in \Omega_{a}$ with $ |\phi_{\la}(te_1)| = |A|$. Further if $A \neq 0$, then there exists distinct $\la_1, \la_2 \in \Omega_{a}$ such that   $|\phi_{\la_1}(te_1)| = |A| = |\phi_{\la_2}(te_1)|$ and $ \phi_{\la_1}(te_1) \neq \phi_{\la_2}(te_1)$.
	
\end{lemma}
\begin{proof}
	Let $A \neq 0$. As  $|A| < \phi_{ia}(te_1),$ and using the fact that map $y \mapsto \phi_{iy}(te_1)$ being analytic and strictly increasing it is possible to find two values $a_1$ and $\ap$ such that $0<a_1 \leq \ap <a$ with $\phi_{i\ap}(te_1) \geq \phi_{ia_1}(te_1) \geq |A|$. By Proposition \ref{prop 2.3}, we have seen that $ |\phi_{\la}(x)| \to 0$ uniformly on ${\Omega_{a}}$, as $|\la| \to \infty$. Thus for each fixed $\ap$ as mentioned above, there exists a $\bt$ such that $|\phi_{i\ap + \bt} (te_1)|=|A|.$ Now the set $\{\la \in \Omega_{a} |~ |\phi_{\la}(te_1)| = |A|\}$ being uncountable and since zeroes of a non-zero analytic function are isolated, there exists $\la_1, \la_2 \in \Omega_{a}$ such that $\phi_{\la_1}(te_1) \neq \phi_{\la_2}(te_1)$ and $|\phi_{\la_1}(te_1)| = |\phi_{\la_2}(te_1)| = |A|$. \\
	The case when $A=0$ follows immediately from the fact that there exists infinitely many $\la \in \R$ which satisfies $\phi_{\la}(te_1) = 0$, once we use the expression for $\phi_{\la}$ involving Bessel functions as given in Subsection \ref{subsection2.1} 
\end{proof}
	
	Equipped with these preliminary results, we proceed to the proofs of our principal theorems.
	\begin{theorem}\label{theorem3.3}
		For a fixed $t>0$ and any $a>0$, let $\{T_k\}$ be a  doubly infinite sequence of exponential type tempered distributions on $\R^d$ satisfying for all $k \in \Z$
		\begin{enumerate}
			\item $M_t T_k = A T_{k+1}$ for some $A \in \C$ and 
			\item $|\langle T_k, \psi \rangle | \leq  M \gamma(\psi)$ for a fixed seminorm $\gamma$ of $S^a(\R^d)$ and a constant $M>0$. 	\end{enumerate}
			Then the following assertions holds. \begin{enumerate}[(a)]
				\item If $|A| = \phi_{ia}(te_1)$, then $\D T_0 = -a^2 T_0$.
				\item If $|A| > \phi_{ia}(te_1)$, then $T_0=0.$
				\item If $|A| < \phi_{ia}(te_1)$, then $T_0$ may not be an eigendistribution of $\D$. If $A$ is also assumed to be non-zero, then $T_0$ may not be an eigendistribution of $M_t$. 
			\end{enumerate}
	
	\end{theorem}
	
	\begin{proof}(a)
		First, we shall consider the case when $T_k$'s are radial. Let $|A|=  \phi_{ia}(te_1)$.  From condition (1) as mentioned in the hypothesis, we obtain $M_t^k T_{-k} = A^k T_0$ and hence $A^{k} \mathcal{H}{T}_0 = \phi_\la(te_1)^k \mathcal{H}{T_{-k}}$. \\The proof shall comprise two steps. In step (1), we shall show that \begin{equation}\label{eq03}
			(M_t- \phi_{ia}(te_1))^{N+1} {T_0} = 0 	\end{equation}  for some $N \in \Z ^+$  or equivalently $ \langle (\phi_\la(te_1) - \phi_{ia}(te_1))^{N+1} \mathcal{H}{T_0}, \phi \rangle = 0 $ for every $\phi \in H_e(\Omega_a)$. Therefore consider for any $\phi \in H_e(\Omega_a)$,
		\begin{align*}
			|\langle(\phi_\la(te_1) - \phi_{ia}(te_1))^{N+1}\mathcal{H}{T_0}, \phi \rangle|  &= \left|\left\langle \mathcal{H}{T_{-k}}, \left(\frac{\phi_\la(te_1)}{A} \right)^k(\phi_\la(te_1) - \phi_{ia}(te_1))^{N+1} \phi \right\rangle\right|\\ &= \left | \left\langle T_{-k}, \mathcal{H}^{-1}\left\{ \left(\frac{\phi_\la(te_1)}{A} \right)^k(\phi_\la(te_1) - \phi_{ia}(te_1))^{N+1} \phi\right\} \right\rangle \right|  \\
			&\leq M \gm \bigg[\mathcal{H}^{-1} \left\{\left(\frac{\phi_\la(te_1)}{A} \right)^k(\phi_\la(te_1) - \phi_{ia}(te_1))^{N+1} \phi\right\} \bigg] \\
			&\leq   M' \mu\bigg[ \left(\frac{\phi_\la(te_1)}{A} \right)^k(\phi_\la(te_1) - \phi_{ia}(te_1))^{N+1} \phi\bigg] 
		\end{align*} where the seminorm $\mu$ is given by $$\mu_{m,\ap}(\phi)= \sup_{\la \in \Omega_{a}^+} \left| \frac{d^\tau}{d\la^\tau} P(\la)\phi(\la)\right|$$ for some even polynomial $P(\la)$ and derivative of even order $\tau$. Let $N = 17\tau + 7$ be fixed. Henceforth, we shall use the following notations. Let $F^k$ denotes the term:  $$\left| \frac{d^\tau}{d\la^\tau} P(\la)\left(\frac{\phi_\la(te_1)}{A} \right)^k(\phi_\la(te_1) - \phi_{ia}(te_1))^{N+1}\phi(\la)\right|$$
		We intend to show that $\sup_{\la\in\Omega_{a}^+} F^k \to 0,$ as $ k \to \infty$. 
		Note that 
		
		\begin{multline} \label{eq02}
			\frac {d^\tau}{d\la^\tau} \bigg(P(\la)\left(\frac{\phi_{\la}(te_1)}{A} \right)^k(\phi_{\la}(te_1) - \phi_{ia}(te_1))^{N+1}\phi \bigg)\\ = \sum_{i+j+l= \tau} C_{ijl} \frac{d^i}{d\la^i} \bigg(\frac{\phi_{\la}(te_1)}{A}\bigg)^k \frac{d^j}{d\la^j} (\phi_{\la}(te_1) - \phi_{ia}(te_1))^{N+1} \frac{d^l}{d\la^l} (P(\la)\phi).
	\end{multline} 
	
		Using $\left| \frac{\phi_\la(te_1)}{A} \right | = \left| \frac{\phi_\la(te_1)}{\phi_{ia}(te_1)} \right| \leq 1 $ for $\la\in\Omega_a$ and $ \phi \in H_e(\Omega_{a}),$ we have for $\la \in \Omega_{a}$, 
		\begin{equation}\label{eq05}
			F^k(\la) \leq C k^{\tau} \left(\frac{\phi_\la(te_1)}{\phi_{ia}(te_1)} \right)^{k-\tau} \left|\frac{\phi_\la(te_1)}{\phi_{ia}(te_1)} - 1 \right| ^{N+1-\tau}
		\end{equation} 
		 for some constant $C$. From Proposition \ref{prop 2.3}, we observed that $|\phi_{\la}(te_1)| \to 0$ uniformly on ${\Omega_{a}}$, as $|\la| \to \infty$ and hence we can find a compact connected neighborhood $V$ of $ia$ in ${\Omega_{a}}$ such that if $\la \notin V$, then $|\phi_\la(te_1)| < \frac{\phi_{ia}(te_1)}{2}$. From \eqref{eq05}, it is clear that $F^k \to 0$ as $k \to \infty$ uniformly on $\Omega \setminus V$. Now we need to establish that 
		\begin{equation}\label{eq06}
			\sup_{\la \in V} k^{\tau} \left( \frac{\phi_{\la}(te_1)}{\phi_{ia}(te_1)}\right)^{k-\tau}  \left|\frac{\phi_\la(te_1)}{\phi_{ia}(te_1)} - 1 \right| ^{N+1-\tau} \to 0.
		\end{equation}
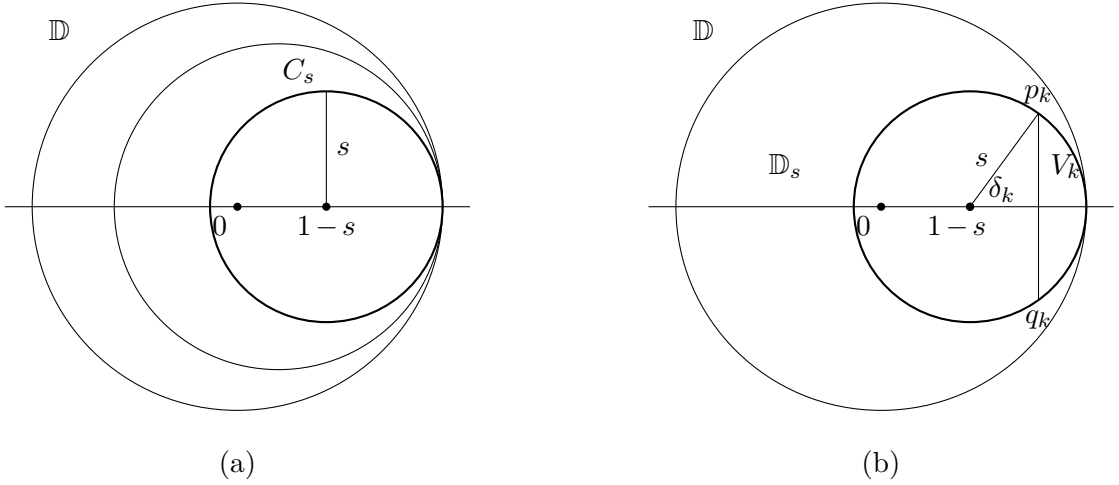
\begin{figure}
		
	\begin{tikzpicture}[scale=1.8, >=latex]
		
		\def\R{1.5}       
		\def\s{0.85}      
		\def\bigS{1.2}    
		\def\gap{4.0}     
		
		\begin{scope}
			\draw[-] (-\R-0.2, 0) -- (\R+0.2, 0);
			
			\draw (0,0) circle (\R);
			\node at (-\R+0.2, \R-0.2) {$\mathbb{D}$};
			
			\draw (\R-\bigS, 0) circle (\bigS);
			
			\coordinate (Center) at (\R-\s, 0);
			\draw[thick] (Center) circle (\s);
			
			\draw (Center) -- ++(0, \s) node[midway, right] {$s$};
			\fill (Center) circle (0.03) node[below] {$1-s$};
			\fill (0,0) circle (0.03) node[below left] {$0$};
			
			\node at ($ (Center) + (-0.2, \s+0.15) $) {$C_s$};
			
			\node at (0, -\R - 0.4) {(a)};
		\end{scope}

	\hspace{0.5in} 

		\begin{scope}[shift={(\gap,0)}]
			\draw[-] (-\R-0.2, 0) -- (\R+0.2, 0);
			
			\draw (0,0) circle (\R);
			\node at (-\R+0.2, \R-0.2) {$\mathbb{D}$};
			
			\coordinate (Center) at (\R-\s, 0);
			\draw[thick] (Center) circle (\s);
			\node at (-\R+0.8, 0.3) {$\mathbb{D}_s$};
			
			\def\dx{0.5} 
			\coordinate (Deltak) at ($ (Center) + (\dx, 0) $);
			
			\pgfmathsetmacro{\y}{sqrt(\s*\s - \dx*\dx)}
			\coordinate (Pk) at ($ (Deltak) + (0, \y) $);
			\coordinate (Qk) at ($ (Deltak) - (0, \y) $);
			
			\draw (Pk) -- (Qk);
			
			\draw (Center) -- (Pk) node[midway, left, xshift=-2pt] {$s$};
			
			\fill (0,0) circle (0.03) node[below left] {$0$};
			\fill (Center) circle (0.03) node[below, xshift=-5pt] {$1-s$};
			\fill (Center) circle (0.03) node[above right, xshift=3pt, yshift=-2pt] {$\delta_k$};

			\node at (Pk) [above] {$p_k$};
			\node at (Qk) [below] {$q_k$};
			\node at ($ (Deltak) + (0.2, 0.3) $) {$V_k$};
			
			\node at (0, -\R - 0.4) {(b)};
		\end{scope}
		
	\end{tikzpicture}
	 \caption{A covering of $ \mathbb{D} \cup \{1\}$  by $\{\mathbb{D}_s \cup \{1\} | 0<s<1\}$, where $\mathbb{D}_s$ are discs bounded by circles $C_s$ of radius $s$ and center $((1-s),0)$}
	\end{figure}
			
			Let $\mathbb{D}$ denotes an open unit disk. Clearly $\mathbb{D} \cup \{1\}$ can be covered by $\{\mathbb{D}_s \cup \{1\} | 0<s<1 \}$ where $\mathbb{D}_s$ is a one-parameter family of open disks of radius $s$ bounded by the circle $C_s : (x-(1-s))^2 + y^2 = s^2 $. One can observe that if $s<s'$, then $\mathbb{D}_s \subset \mathbb{D}_{s'}$. Now using the facts that $V$ being compact and connected and since $\left|\frac{\phi_{\la}(te_1)}{\phi_{ia}(te_1)} \right| <1$ for $ \la \in \Omega_a, \la \neq \pm ia$, the image of $V$ under the map $\la \mapsto \frac{\phi_{\la}(te_1)}{\phi_{ia}(te_1)}$ is a connected set contained inside the set $\mathbb{D}_s \cup \{1\}$ for some $0<s<1$. Let us assume that $0 \in \mathbb{D}_s.$ In view of \eqref{eq06}, it suffices to prove that $ k^{\tau} z^{k-\tau} |z-1| ^{N+1-\tau} \to 0$, uniformly as $k \to \infty$ on   
		$\mathbb{D}_s \cup \{1\}$.
		Let $p_k$ and $q_k$ denotes two points on $C_s$ given by
		\[
		p_k = (1 - s) + s e^{i\delta_k} \quad \text{and} \quad q_k = (1 - s) + s e^{-i\delta_k},
		\]
		for some $0 < \delta_k < \pi/2$ so that $s - s \cos \delta_k = k^{-1/4}$. Let $V_k$ be the intersection of $\mathbb{D}_s \cup \{1\}$ with the minor circular segment of width $k^{-1/4}$ joining the points $p_k$, $q_k$ and $1$ (as shown in Fig.~1(b)). Precisely,
		\[
		V_k = \{ z \in D_s \mid |1 - \Re(z)| < k^{-1/4} \} \cup \{1\}.
		\]
		
		Note that, $|p_k|^2 = 1 - (2s - 2s^2)(1 - \cos(\delta_k)) = 1 - 2(1 - s) k^{-1/4}$
	and	\\ $|p_k - 1|^2 = \left| 1 - \left( (1 - s) + s e^{i \delta_k} \right) \right|^2 = 2s k^{-1/4}.$ For $z \in \mathbb{D}_s \setminus V_k$, we have $ |z| \leq |p_k| $ and hence for some constant $C'$,
		\begin{align}
			k^\tau |z|^{k - \tau} |z - 1|^{N + 1 - \tau}  \notag
			&\leq C' k^\tau |p_k|^{k - \tau} \\
			&= C' k^\tau \left( 1 - 2(1 - s) k^{-1/4} \right)^{k - \tau},\label{eq:vk1}
		\end{align}
		For $ z \in V_k $, we have $|z - 1| \leq |p_k - 1| $ and therefore for some constants $C'$ and $ C'' $,
		\begin{align}
			k^\tau |z|^{k - \tau} |z - 1|^{N + 1 - \tau} \notag
			&\leq C' k^\tau |p_k - 1|^{N + 1 - \tau} \\  \notag
			&= C' k^\tau \left( 2s k^{-1/4} \right)^{N + 1 - \tau} \\  
			&\leq C'' k^{-(\tau + 1)}.\label{eq:vk2}
		\end{align}
		
		It is now immediate from \eqref{eq:vk1} and \eqref{eq:vk2} that as $ k \to \infty,$ $$
		\sup_{z\in \mathbb{D}_s \cup \{1\}} k^\tau |z|^{k - \tau} |z - 1|^{N + 1 - \tau} \to 0.$$
		Thus we have	$\left( \phi_\la(te_1) - \phi_{ia}(te_1) \right)^{N + 1} T_0 = 0$, hence $(M_t - \phi_{ia}(te_1)))^{N + 1} T_0 = 0.$
		Finally from Lemma \ref{lemma3.1}, we obtain $ M_t T_0 = \phi_{ia}(te_1) T_0.$
		
	Next we shall prove the theorem when $T_k$'s are non-radial. For clarity and to allow this argument to extend to other multipliers in later sections, we denote the multiplier operator $M_t$ by $\Theta$ and consider it as a multiplier from 
	$S^a(\R^d)$ to itself with symbol $m(\la)=\phi_\la(te_1)$. With this notation, we have already established that $\Theta T_0= \phi_{ia}(te_1) T_0$, when $T_k$ are radial. 

 We now consider the case of a general family $\{T_k\}_{k \in \Z}$. We will show that if this sequence satisfies the hypothesis, then for each $y \in \R^d$, the sequence $ \{R \ell_yT_k\}$  also satisfies the same hypothesis. Given that $\Theta$ commutes with radialization and translations, the hypothesis $ \Theta T_k = AT_{k+1}$ immediately yields $ \Theta R(\ell_{y} T_{k}) = A R(\ell_{y} T_{k+1}).$ It remains to show that for the seminorm $\gm$ of $S^a(\R^d)$ as seen in the hypothesis of the theorem and $\psi \in \sar(\R^d)$, the following inequality holds $|\langle R( \ell_{y}T_k), \psi \rangle | \leq  M \gamma(\psi) $. For that, consider any $\psi \in \sar(\R^d)$, we have \begin{align*}
	\gm(\ell_{y}\psi) &= \sup_{x \in \R^d} (1+|x|)^m e^{a|x|} D^{\ap} \psi(x-y) \\ &\leq \sup_{z \in \R^d} (1+|y|+|z|)^m e^{a(|y| + |z|)} D^{\ap} \psi(z)\\ & = (1+|y|)^m e^{a|y|} \sup_ {z \in \R^d} (1+ |z|)^m e^{a|z|} D^\ap \psi(z) \\ &\leq C_y \gm (\psi)
\end{align*}
Using the fact that $|\langle T_k, \psi \rangle| \leq M \gm(\psi) $ for any $ \psi \in S^a(\R^d)$, it follows that for any $ \psi_1 \in \sar(\R^d)$,
$$ |\langle R(\ell_{y}T_k), \psi_1 \rangle| = | \langle \ell_{y}T_k, \psi_1\rangle| = | \langle T_k, \ell_{-y} \psi_1 \rangle| \leq M \gm (\ell_{-y} \psi_1) \leq C_{y^{-1}} M \gm(\psi_1) $$ 
Combining the result proved for radial distributions with the properties of the multiplier $\Theta$, we deduce that 
$$ \Theta R(\ell_{y}(T_{0})) = m(ia) R(\ell_{y}T_0) , \quad{for~all~} y \in \R^d $$
Recalling that $\Theta$ commutes with both translations and radialization, we obtain
$R(\ell_{y}(\Theta T_0)) = R \ell_{y} (m(ia)T_0)$, for all $y \in \R^d$. From the proof of Proposition \ref {theorem2.15}, recall the fact that for any $T\in S^a(\R^d)'$, if $R(\ell_xT)=0$ for all $x \in \R^d$, then $T$ equals zero as a distribution.
Hence we proved $ \Theta T_0 = m(ia) T_0 $. Finally by One Radius Theorem $\D T_0 = -a^2 T_0$.
	\end{proof}
	
Proof of Part (b).
	We have for $\phi \in H_e(\Omega_a)$\\
	\begin{equation}
		|\langle \mathcal{H}{T_0}, \phi \rangle| \leq M' \mu \bigg[\left(\frac{\phi_{\la}(te_1)}{A}\right)^k \phi \bigg]
	\end{equation}
	 The hypothesis and Proposition \ref{Prop2.2} together implies   $ |\phi_{\la}(te_1)| \leq \phi_{ia}(te_1) < |A|$ for all $\la \in \Omega_a$. Therefore, $\mu \bigg[\left(\frac{\phi_{\la}(te_1)}{A}\right)^k \phi \bigg] \to 0,$ as $k \to \infty$, which implies that $ | \langle \mathcal{H} T_0, \phi \rangle| =0,$ for every $\phi \in H_e(\Omega_a)$. Hence $T_0=0.$
	
	Proof of Part (c).
	For $|A| < \phi_{ia}(te_1)$, Lemma \ref{lemma3.2} guarantees the existence of infinitely many $\la$'s in $\Omega_{a}$ such that $|\phi_\la(te_1)| = |A|$. Let $\la_1$ and $\la_2$ be two such distinct $\la$s with $\la_{1} \neq \pm \la_{2}$. Let $\phi_{\la_1}(te_1)=A e^{i \theta_1} $ and $\phi_{\la_2}(te_1)=Ae^{i \theta_2}$. 	Setting a sequence of distributions in the form $T_k = e^{ik \theta_1} \phi_{\la_1} + e^{ik \theta_2} \phi_{\la_2}$, where $k \in \Z$, one can easily verify that the sequence $\{T_k\}_{k \in \Z}$ satisfies the hypothesis of the theorem, but $T_0$ fails to be an eigendistribution of $\D$.  
	
	If $A\neq 0 $, on choosing $\la_{1}$ and $\la_{2}$ such that $|\phi_{\la_1}(te_1)|=|A|=|\phi_{\la_2}(te_1)|$, but $\phi_{\la_1}(te_1) \neq \phi_{\la_2}(te_1)$ and substituting in the above example would help us obtain the required counter-example in this case. Again Lemma \ref{lemma3.2} ensures the existence of such $\la$s.
	\subsubsection{Proof of Theorem \ref{theorem1.5}}
	 As proved in \cite{BP}, we shall use the fact that the functions $f_k$ on $\R^d$ satisfying the property that  $ |f_k(x)| \leq M e^{a |x|} $ for a constant $M>0$ and for every $x \in \R^d$ are themselves exponential tempered distributions of type $a$ satisfying the hypothesis of Theorem \ref{theorem3.3}. Thus Theorem \ref{theorem1.5} is an immediate consequence of Theorem \ref{theorem3.3}.

	Analogously, we shall formulate an equivalent version of  Theorem \ref{theorem3.3} for the volume average $B_t$ as follows.
		\begin{theorem}\label{theorem3.4}
		For a fixed $t>0$ and any $a>0$, let $\{T_k\}$ be a  doubly infinite sequence of exponential tempered distributions of type $a$ on $\R^d$ satisfying for all $k \in \Z$
		\begin{enumerate}
			\item $B_t T_k = A T_{k+1}$ for some $A \in \C$ and 
			\item $|\langle T_k, \psi \rangle | \leq  M \gamma(\psi)$ for a fixed seminorm $\gamma$ of $S^a(\R^d)$ and a constant $M>0$. 	\end{enumerate}
		Then the following assertion holds. 
		\begin{enumerate}[(a)]
			\item If $|A| = \psi_{ia}(te_1)$, then $\D T_0 = -a^2 T_0$.
			\item If $|A| > \psi_{ia}(te_1)$, then $T_0=0$
			\item If $|A| <  \psi_{ia}(te_1)$, then $T_0$ may not be an eigendistribution of $\D$. If $A$ is also assumed to be non-zero, then $T_0$ may not be an eigendistribution of $B_t$. 
		\end{enumerate}
		
	\end{theorem}
	
	Proof of the above formulated theorem follows the same line of argument as in the proof of Theorem \ref{theorem3.3} with appropriate change in symbols. On invoking the same logic we have mentioned in the proof of Theorem \ref{theorem1.5}, we obtain the following corollary.
	\begin{corollary}\label{corollary3.5}
		For a fixed $t>0$ and any $a>0$, let $\{f_k\}_{k \in \Z}$ be a doubly infinite sequence of functions on $\R^d$ satisfying for all $k \in \Z$
		\begin{enumerate}
			\item $B_t f_k = A f_{k+1}$ for some $A \in \C$ and 
			\item 	$ |f_k(x)| \leq M e^{a |x|} $ for a constant $M>0$ and for every $x \in \R^d$. \end{enumerate} Then the following assertion holds.
		\begin{enumerate}[(a)]
			\item If $|A| = \psi_{ia}(te_1)$, then $\D f_0 = -a^2 f_0$.
			\item If $|A| > \psi_{ia}(te_1)$, then $f_0=0.$
			\item If $|A| < \psi_{ia}(te_1)$, then $f_0$ may not be an eigenfunction of $\D$. 
		\end{enumerate}
		
	\end{corollary} 
	Finally we shall state an analogous result for the heat operator $e^{-t\D}$ as below.
	
		\begin{theorem}\label{theorem3.6}
		For a fixed $t>0$ and any $a>0$, let $\{T_k\}$ be a doubly infinite sequence of exponential tempered distributions of type $a$ on $\R^d$ satisfying for all $k \in \Z$
		\begin{enumerate}
			\item $e^{-t\D} T_k = A T_{k+1}$ for some $A \in \C$ and 
			\item $|\langle T_k, \psi \rangle | \leq  M \gamma(\psi)$ for a fixed seminorm $\gamma$ of $S^a(\R^d)$ and a constant $M>0$. 	\end{enumerate}
		Then the following conclusion holds.
		 \begin{enumerate}[(a)]
			\item If $|A| = e^{ta^2}$, then $\D T_0 = -a^2 T_0$.
			\item If $|A| > e^{ta^2}$, then $T_0=0.$
			\item If $0 \neq |A| <  e^{ta^2}$, then $T_0$ may not be an eigendistribution of $e^{-t \D}$. 
			\item If $A=0$, then each $T_k = 0$ for all integer k.
		\end{enumerate}
			\end{theorem}
		\begin{proof}
			We omit the proof of Part (a) and Part (b) as it requires to emulate the same arguments given in the corresponding proof of Theorem \ref{theorem3.3}.
			
			(c)  	Let $|A| \neq 0$. As  $|A| < e^{ta^2} ,$ it is possible to find an $a_0$ such that $0<a_0<a$ with  $e^{ta_0^{2}} \geq |A| $. Since $e^{-t\la^2}\to 0$ as $|\la|\to \infty$ for $\la \in \Omega_a$, for each fixed $\bt$ with $0<a_0\leq \bt<a$, we have uncountably many $\la \in \C$ with $|\im(\la)|=\bt$ and $|e^{-t\la^2}|=|A|$.

 Now the zeroes of a non-zero analytic function are isolated, there exists $\la_1, \la_2 \in \Omega_{a}$ such that $ e^{-t\la_{1} ^2} \neq e^{-t\la_{2} ^2}$ still $|e^{-t\la_{1} ^2}| = |e^{-t\la_{2} ^2}| = |A|$. Let $e^{-t\la_{1} ^2} = A e^{i\theta_1}$ and $e^{-t\la_{2} ^2} = A e^{i\theta_2}$. Now setting a sequence of distributions of the form $T_k = e^{ik\theta_1}  \phi_{\la_1} +  e^{ik\theta_2} \phi_{\la_2} $ for $k \in \Z$, we obtain a sequence $\{T_k\}$ satisfying the hypothesis of the theorem, but $T_0$ is not an eigendistribution of $e^{t\D}$.
			
			(d) We aim to show that if the heat semigroup acting on a tempered distribution 
			$T$ yields zero ($e^{-t \D} T=0$),
			 then the distribution T itself must be zero.
			 On the contrary, let us assume that $T \neq 0$. As in the proof of Theorem \ref{theorem2.15}, we may assume $T$ to be radial. Since $T \neq 0$, there exists a $\psi \in \sar(\R^d)$ such that $\langle T, \psi \rangle \neq 0$. However the condition $e^{-t \D} T=0$ implies that $ \langle T, \psi * h_t \rangle =0 $ for every $\psi \in \sar(\R^d)$. Now by proposition \ref{prop2.9}, $\{\psi * h_t \mid \psi \in \sar(\R^d)\}$ is dense in $\sar(\R^d)$ which implies that $\langle T,  \psi \rangle =0$ for all $\psi \in \sar(\R^d)$. Hence $T=0$.
		\end{proof}
	As a consequence of the above theorem we have the following corollary.	

	\begin{corollary}\label{corollary3.7}
		For a fixed $t>0$ and any $a>0$, let $\{f_k\}_{k \in \Z}$ be a doubly infinite sequence of functions on $\R^d$ satisfying for all $k \in \Z$
		\begin{enumerate}
			\item $e^{-t\D} f_k = A f_{k+1}$ for some $A \in \C$ and 
			\item 	$ |f_k(x)| \leq M e^{a |x|} $ for a constant $M>0$ and for every $x \in \R^d$. \end{enumerate} Then the following assertions holds.
		\begin{enumerate}[(a)]
			\item If $|A| = e^{ta^2}$, then $\D f_0 = -a^2 f_0$.
			\item If $|A| > e^{ta^2}$, then $f_0=0.$
			\item If $|A| < e^{ta^2}$, then $f_0$ may not be an eigenfunction of $\D$. 
		\end{enumerate}
		
	\end{corollary} 
	
	\section{Multipliers on $\mathcal{S}(\R^d)$} \label{section4} 
	Building on our detailed analysis on eigenfunctions of three specific multipliers in the space $S^a(\R^d)$, we will extend the previous results to all possible multipliers on Schwartz space in $\R^d$. The one-dimensional nature of the spectrum $[0, \infty)$ for the Laplacian in $\mathcal{S}(\R^d)$, allows us to formulate more general results corresponding to any arbitrary multiplier $\Theta$. We will be treating multipliers having real valued symbols and complex valued symbols separately. Let us begin with a theorem that determines support of $\widehat{T_0}$ where  $\{T_k\}$ is a sequence of tempered distributions satisfying the Strichartz criterion.
	\begin{theorem}\label{theorem4.1}
	Let $\Theta$ be a multiplier on $\mathcal{S}(\mathbb{R}^d)$ with symbol $m(\xi) \in \mathcal{C}^{\infty}(\R^d)$.
	Suppose $\{T_k\}_{k \in \Z}$ be a doubly infinite sequence of tempered distributions on $\mathbb{R}^d$ such that for all $k \in \mathbb{Z}$, $|\langle T_k, \psi \rangle| \leq M \gamma(\psi)$ for a fixed seminorm $\gamma$ on $\mathcal{S}(\mathbb{R}^d)$ and a constant $M>0$.
	Let $A \in \mathbb{C}$. Then the following conclusions hold.
	\begin{enumerate}[(i)]
		\item If for all $k \in \mathbb{Z}^+$, $\Theta T_k = A T_{k+1}$, then $\mbox{Supp}\; \widehat{T}_0 \subseteq \{ \xi \in \mathbb{R}^d \mid |m(\xi)| \leq |A| \}$.
		\item If for all $k \in \mathbb{Z}^-$, $\Theta T_{k-1} = A T_k$, then  $Supp\; \widehat{T}_0 \subseteq \{ \xi \in \mathbb{R}^d \mid |m(\xi)| \geq |A| \}$.
		\item If for all $k \in \mathbb{Z}$, $\Theta T_k = A T_{k+1}$, then $ Supp\; \widehat{T}_0 \subseteq \{ \xi \in \mathbb{R}^d \mid |m(\xi)| = |A| \}$.
	\end{enumerate}
\end{theorem}

	\begin{proof}
	Let $\Omega = \{ \xi \in \mathbb{R}^d \mid |m(\xi)| \leq |A| \}$ and choose $\beta \notin {\Omega}$. Then it is possible to find a positive constant $c < 1$ and a compact neighborhood $V$ of $\beta$ such that
	$|A| < c |m(\xi)|$ for all $\xi \in V.$
Let $\phi \in \mathcal{S}(\R^d)$ such that Supp $\phi \subset V$. It suffices to prove that $\langle \widehat{T}_0, \phi \rangle = 0$.
	
 Since $\phi$ is compactly supported and $m \in C^{\infty}(\R^d)$, we have
	$$\phi(\xi) A^k / m(\xi)^k \in \mathcal{S}(\R^d).$$
Suppose that for $\psi \in \mathcal{S}(\R^d)$ such that $\ds {\psi}(x) = A^k \Big(\frac{\phi}{ m(\xi)^k}\hat{\Big)}(x)$.
	
 Using the hypothesis, we obtain for all $k \in \Z^+$, $\Theta^k T_{0} = A^k T_k$, equivalently $m(\xi)^k \widehat{T}_0 = A^k \widehat{T}_k$.
Therefore \begin{align*}
	|\langle \widehat{T}_0, \phi \rangle| &= \left| \left\langle \widehat{T}_k, \frac{A^k}{m(\xi)^k} \phi \right\rangle \right| = |\langle T_k, \psi \rangle|\\
	& \leq M \gm(\psi) \leq M' \mu_{\ap,\bt} \left[ \left( \frac{A}{m(\xi)} \right)^k \phi \right]
\end{align*}
	
	where $\mu_{\alpha, \beta}$ is a seminorm given by
	\begin{align*}
		\mu_{\alpha, \beta} \left[ \left( \frac{A}{m(\xi)} \right)^k \phi(\xi) \right] &= \sup_{\xi \in V} \left| \xi^\ap D^\bt \phi(\xi) \left(\frac{A}{m(\xi)}\right)^k \right| \\
		&= \sup_{\xi \in V} \left | \xi^\alpha \sum_{\tau \leq \beta} \binom{\beta}{\tau} D^\tau \phi(\xi) D^{\beta - \tau} \left(\frac{A}{m(\xi)}\right)^k \right| 
	\end{align*}
	Since $\phi \in C_c^\infty(\mathbb{R}^d)$, it is easy to see that $\xi^\alpha D^\tau \phi(\xi)$ is bounded in V for every multi-indices $\alpha, \tau $.
	We note that $$ \ds \sum_{\tau \leq \beta}  D^{\beta - \tau} \left(\frac{A}{m(\xi)}\right)^k \leq C \sum_{\tau \leq \beta}  \left(\frac{A}{m(\xi)}\right)^{k - |(\beta - \tau)|}.$$
	Since $ \ds \left|\frac{A}{m(\xi)}\right| < c$, $$ \ds \sup_{\xi \in V} \left|\sum_{\tau \leq \beta} \xi^\alpha \left(\frac{A}{m(\xi)}\right)^{k - |(\beta - \tau)|}\right| \to 0$$ as $k \to \infty$ for all multi-indices $\alpha$ and $ \beta$.
	Hence $ \ds\mu_{\alpha, \beta} \left[ \left(\frac{A}{m(\xi)}\right)^k \phi(\xi) \right] \to 0$, as $k \to \infty$.
	Thus $\langle \widehat{T_0}, \phi \rangle = 0$ as $k \to \infty$ and hence $\beta \notin \text{Supp} \hat{T_0}$.
	Therefore $\text{Supp} \widehat{T_0} \subset \Omega$.
	This completes the proof of (i).
	Now adapting a similar argument with negative integers and tending $k \to -\infty$ yields (ii). And (iii) follows trivially from (i) and (ii).
\end{proof}	
Following is one of the central result in this section where we establishes a Strichartz's type criterion for an arbitrary multiplier $\Theta$ having real valued symbols on $\mathcal{S}(\R^d)$.  
\begin{theorem}\label{theorem4.2}
	Let $\Theta: \mathcal{S}(\R^d) \rightarrow \mathcal{S}(\R^d)$ be a multiplier with real valued symbol $m(\xi) \in C^\infty (\R^d)$. Let $\left\{T_k\right\}$ be a doubly infinite sequence of tempered distributions on $\R^d$. Suppose that for all $ k \in \Z$, $\Theta T_k=A T_{k+1}$ for a non-zero constant $A \in \mathbb{C}$ and $\left|\langle T_k, \psi \rangle \right| \leq M \gamma(\psi)$ for a fixed seminorm $\gamma(\psi)$ of $\mathcal{S}(\R^d)$ and a constant $M >0$. Let $m(\R^d)=\left\{m(\xi)| \xi \in \R^d\right\}$.
	\begin{enumerate}[(a)]
		\item  If $|A| \in m(\R^d)$, but $-|A| \notin m(\R^d)$, then $ \Theta T_0=|A| T_0$.
		\item If $-|A| \in m(\R^d)$, but $|A| \notin m(\R^d)$, then $ \Theta T_0=-|A| T_0$.
		\item   If both $|A|$, $-|A| \in m(\R^d)$, then $T_0$ can be uniquely written as $T_0=T_{+}+T_{-}$, where $T_{+}, T_{-} \in S(\R^d)^{\prime}$ satisfying $\Theta T_{+}=|A| T_{+}$ and $ \Theta T_{-}=-|A| T_{-}$.
		\item 	If neither $|A|$ nor $-|A| $ is in $ \R^d$, then $T_0=0$.
	\end{enumerate}
	\end{theorem}
		\begin{proof}
		(a)
	Let $|A|=m(\alpha)$ for some $\ap \in \R^d$.
	It suffices to show that $(\Theta-m(\alpha))^{N+1} T_0=0$ for some $N \in \mathbb N$ or equivalently $\left\langle(m(\xi)-m(\alpha))^{N+1} \hat{T}_0, \phi\right\rangle=0$, for all $\phi \in C_c^{\infty}(\R^d)$.
	Combining the facts that  $m$  is real valued and $-|A| \notin m( \R^d)$ with theorem \ref{theorem4.1}(iii), we conclude that
	$$\text{Supp~} \widehat{T_0} \subset\left\{\xi \in \mathbb{R}^d| | m(\xi)|=|A|\right\} =\left\{\xi \in \mathbb{R}^d| m(\xi)=m(\alpha)\right\}. $$
	Let ` $g$ ' be an even function in $C_c^{\infty}(\mathbb{R})$ such that
	$g \equiv 1 \text { on }[-1 / 2,1 / 2] \text { \& Supp} (g) \subseteq(-1,1) $.
	For $r>0$, let $g_r$ be defined by $g_r(\lambda) = g(\frac{\lambda}{r})$. 
	Let $B=\max \left\{\left|\frac{d^k}{d \lambda^k} g(\lambda)\right| ; \lambda \in(-1,1), k \leq N\right\}$.
	Hence we obtain $\left|\frac{d^k}{d \lambda^ k} g_r(\lambda)\right| \leq \frac{B}{r^k}$ for all $k \leq N$.
	Let $p(\xi) =m(\xi)-m(\ap)$. For $\phi \in C_c^{\infty}(\R^d)$, define $\mathfrak{H}_r(\xi) = (m(\xi) - m(\ap))^{N+1} g_r(p(\xi)) \phi(\xi)$. It can be verified that $\mathfrak{H}_r \in \mathcal{S}(\R^d)$. Let $h_r \in \mathcal{S}(\R^d)$ such that ${h_r}=\hat{ \mathfrak{H}}_r.$ Since $\mathfrak{H}_r(\xi) = (m(\xi) - m(\ap))^{N+1} \phi(\xi)$ in a neighborhood of Supp $(\widehat{T_0})$, we have $$|\langle (m(\xi) - m(\ap))^{N+1} \widehat{T_0}, \phi \rangle| = | \langle \widehat{T_0}, \mathfrak{H_r} \rangle| = |\langle T_0, h_r \rangle| \leq M \gm(h_r) \leq M' \mu_{\ap,\bt}(\mathfrak{H_r}),$$ where the seminorm $\mu_{\ap,\bt}(\mathfrak{H}_r)= \sup_{\xi \in \R^d} |\xi^\ap D^\bt \mathfrak{H}_r(\xi)|$. The proof will be complete once we show that $\mu_{\ap,\bt}(\mathfrak{H}_r) \to 0$, as $r \to 0$ when $\mathfrak{H}_r$ is defined for a suitably large $N$. As $g_r(p(\xi))$ vanishes outside the set $\{\xi \in \R^d | |p(\xi)| > r\}$, it is enough to consider the supremum over the set $$ E = \{\xi \in \R^d | |p(\xi)| \leq r\} \cap \text{Supp} \;\phi.$$ We note that for $\xi \in E,$ we have $ |(m(\xi) - m(\ap)| \leq r.$ Now \begin{align*}
		\mu_{\ap,\bt}(\mathfrak{H_r}) & = \sup_{\xi \in E} \left| \xi^\ap D^\bt (m(\xi) - m(\ap))^{N+1} g_r(p(\xi)) \phi(\xi) \right| \\ & =  \sup_{\xi \in E} \left| \xi^\ap \sum_{\eta, \delta,\tau \leq \bt} C_{\eta, \delta, \tau} D^\eta (m(\xi) - m(\ap))^{N+1} D^\delta \phi(\xi) D^\tau g_r(p(\xi)) \right| \\ & \leq M \sup_{\xi \in E} \sum_{\eta+\tau \leq \bt}  C_{\eta, \tau} \left| D^\eta (m(\xi) - m(\ap))^{N+1} \right| \left| D^\tau g_r (p(\xi)) \right| \\ & \leq M_1 \sum_{\eta+\tau \leq \bt} r^{N+1-|\eta|} r^{-|\tau|} 
	\end{align*}
Thus for $r \to 0,$ it can be seen $\mu_{\ap,\bt}(\mathfrak{H_r})$ tends to zero. Hence $\left\langle(m(\xi)-m(\alpha))^{N+1} \hat{T}_0, \phi\right\rangle=0$ for large $N > |\bt|$. On invoking Lemma 3.1, we obtain $(\Theta-m(\ap)) T_0 =0$.

\noindent (b) It can be seen that $-\Theta$ satisfies the hypothesis of Part(a) with $A$ replaced by $-A$. Hence it follows from (a) that $\Theta T_0 = -|A|T_0$.

\noindent (c) By setting $\Theta_{0} = \Theta^2$, it is clear that the sequence $\{T_{2k}\}$ satisfies the recurrence relation $\Theta_{0} T_{2k} = A^2 T_{2k+2}$.  This allows us to apply the result of part (a) for the parameters $\Theta_{0}$ and $|A|^2$, yielding the relationship: $\Theta_{0} T_{0} = |A|^2 T_{0}$. 

Now we define the components $T_{+}$ and $T_{-}$ explicitly as follows: 
$$T_{+} = \frac{|A| T_{0} + \Theta T_{0}}{2|A|}$$ 
and $$ T_{-} = \frac{|A| T_{0} - \Theta T_{0}}{2|A|}.$$
It can be easily verified that $T_0 = T_+ + T_-$ and that these components satisfies the required properties $\Theta T_+ = |A| T_+$ and $\Theta T_- =- |A| T_-$.

To prove uniqueness, assume an alternative decomposition of the form $T_{0} = S_{+} + S_{-}$ exists, where $S_{+}$ and $S_{-}$ also satisfy the required eigenvalue properties. i.e. $ \Theta S_{+} = |A| S_{+}$ and $ \Theta S_{-} = -|A| S_{-}$. We then have a system:$T_{0}	 = S_{+} + S_{-}$ and $\Theta T_{0} = |A| S_{+} - |A| S_{-}$. Solving this linear system for $S_{+}$ and $S_{-}$ uniquely determines them to be $S_{+} = T_{+}$ and $S_{-} = T_{-}$, confirming the decomposition is unique.

\noindent (d) From Theorem \ref{theorem4.1}, it is straight forward to see that the distributional support of $T_{0}$ is empty for all $x \in \R^d$.  Hence $T_{0} = 0$.  
\end{proof}
The preceding theorem yields several interesting corollaries when specialized to particular multipliers. We list a few representative cases for brevity.

Let $\{f_k\}_{k \in \mathbb{Z}}$ be a doubly infinite sequence of measurable functions on $\R^d$ satisfying $\Theta f_k = A f_{k+1}$ for all $k \in \mathbb{Z}$, where $\Theta$ is a multiplier (specified below) and $A \in \mathbb{C}$ is a constant. The following conclusions hold:
\begin{enumerate}[(1)]
\item	For $\Theta= \D$, suppose that for all $k \in \Z$ such that $\|f_k\|_{p} \leq M$ for $\frac{2d}{d-1} < p \leq \infty $ and for any $A \in \C$, then $ \D f_0 = |A| f_0$.

\item For $\Theta = e^{-t\D}$, for a fixed $t>0$, suppose that  $k \in \Z$ such that $\|f_k\|_{p} \leq M$ for $\frac{2d}{d-1} < p \leq \infty $ and $|A| = e^{-t \la_{0}^2}$ for some $\la_{0} \in \R$, then $ \D f_0 = \la_{0}^2 f_0$.

\item For $\Theta = M_t$(respectively $\Theta=B_t$), for a fixed $t>0$, suppose that  $k \in \Z$ such that $\|f_k\|_{p} \leq M$ for $1 \leq p \leq \infty $ and $|A| = 1$, then $ \D f_0 = 0$.
\end{enumerate}

Note that the space $L^p(\R^d)$ contains eigenfunctions of $\D$ only for $p> \frac{2d}{d-1}$.
The key ingredients behind the proof of the above assertions are as follows:

\begin{enumerate}[(a)] \item It is well known that functions in $L^p$ for $1\leq p \leq \infty $ are tempered distributions.
	\item  Note that $\D$ when considered as a multiplier has a symbol that takes only positive values and hence the conclusion follows by Theorem \ref{theorem4.2} (i). 
	\item A similar argument when applied to the heat operator $e^{-t\D}$ (respectively $M_t$ and $B_t$) whose symbol assumes only positive values on $\R$, would yield us $e^{-t\D} f_0= e^{-t\la_{o}^2}f_0$ (respectively $M_t f_{0} =f_{0}$ and $B_t f_{0} =f_{0}$). On further applying the One Radius Theorem, we obtain the required conclusions.
	
\end{enumerate}
In contrast to the preceding discussion which addressed multipliers with exclusively real symbols, we will now analyze those with complex-valued symbols. 
\begin{theorem}\label{theorem4.3}
	Let $\{f_k\}_{k \in Z}$ be a doubly infinite sequence of functions on $\R^d$ and $\Theta $ be a multiplier on $\mathcal{S}(\R^d)$ with symbol $m(\xi) \in C^\infty (\R^d)$. Suppose that for every $k \in \Z, \|f_k\|_p \leq M $ for $1 \leq p \leq \infty$ and $\Theta f_k = A f_{k+1}$ for constants $M>0$, $A \in \C$. If $\{\xi \in \R^d| |m(\xi)| = |A|\}$ is finite, then $f_0$ can be uniquely written as $f_0 = g_1+ g_2+ \cdots+ g_r$ for functions $g_i, i=1,2, \ldots, r$ on $\R^d$ satisfying $\Theta g_i= A_ig_i$, where $A_i \in \C$ are distinct and $|A_i| = |A|$. Further if $m(\xi)= c$, a constant for all $\xi \in E$, then $f_0$ is an eigenfunction of $\Theta$ with eigenvalue $c$.
	\end{theorem}
	To prove Theorem \ref{theorem4.3}, it is necessary to establish two supporting lemmas and a key proposition, which are given below. In what follows $E_{\zeta} (x)= e^{i x \cdot \zeta}$ for $x, \zeta \in \R^d$.  $\pa^{\ap}_{\zeta} ~\text{denotes}~ \frac{\pa^{|\ap|}}{\pa^{\ap_{1}} _{\zeta_{1}} \pa^{\ap_{2}} _{\zeta_{2}} \ldots \pa^{\ap_{d}} _{\zeta_{d}}}.$ 
	\begin{lemma}\label{lemma4.4}
		Let $ \zeta^{(1)}, \zeta^{(2)}, \ldots, \zeta^{(k)}$ be distinct vectors in $\R^d$ and $P_1, P_2, \ldots, P_k$ be polynomials in $d$ variables. Suppose that $\sum_{j=1}^{k} P_j(\partial_{\zeta}) E_{\zeta}|_{\zeta=\zeta^{(j)}} \in L^p(\R^d)$ for $1 \leq p \leq \infty$, then the following assertions are satisfied.
		\begin{enumerate}[(1)]
                     \item If $1 \leq p < \infty$, then $P_j$'s are zero polynomials.
		\item  If $p = \infty$, then $P_j, j= 1, \ldots, k$ are constant polynomials.
		
		\end{enumerate}
	\end{lemma}
	
	\begin{proof}
		We will give the proof of the lemma for $d=1$ and similar techniques can be emulated to extend the proof to higher dimensions. We assume that $$f = \sum_{i=1}^{k} P_i(\pa_\zeta)E_{\zeta} \Big|_{\zeta= \zeta^{i)}}.$$ Without loss of generality, suppose that $P_j$ is a non-zero polynomial of degree $n_j,$ for $ j = 1,2, \ldots k$. Let $N_i = \max_{j \neq i} n_j. $
We choose a function $\psi \in \mathcal{S}(\R)$ such that $ \partial_{\zeta}^{(j)} \hat{\psi}|_{\zeta = \zeta^{(i)}} =0$ for every $2 \leq i \leq k, 0 \leq j \leq N_{1}$, and $\partial_{\zeta}^{j} \hat{\psi}|_{\zeta = \zeta^{(1)}} =0,$ for all $0 \leq j <  n_1$ and $\partial_{\zeta}^{n_1} \hat{\psi}|_{\zeta = \zeta^{(1)}} \neq0$.

	Then \begin{align*} f * \psi &= \sum_{i=1}^{k} P_i(\partial_{\zeta}) (E_{\zeta}* \psi)|_{\zeta=\zeta^{(i)}} = \sum_{i=1}^{k} P_i(\partial_{\zeta}) (\hat{\psi(\zeta)} E_{\zeta})|_{\zeta=\zeta^{(i)}} \\  &=C \pa_{\zeta}^{n_{1}} \hat{\psi}|_{\zeta=\zeta^{(1)}} E_{\zeta^{(1)}} = C \pa_{\zeta}^{n_{1}} \hat{\psi}{\zeta^{(1)}} E_{\zeta^{(1)}}.
		 \end{align*} where $C$ be any arbitrary constant.  Since $f \in L^p(\R)$ and $\psi \in \mathcal{S}(\R), f*\psi \in L^p(\R)$ and hence $E_{\zeta^{(1)}} \in L^p(\R)$ which is a contradiction. Therefore, $P_1$ is a zero polynomial. On suitably choosing $\psi$, we can show that $P_j$, for $2 \leq j \leq k$ are also zero polynomials. This completes the proof of Part (1).
		 
		 	Further using the fact that $\pa_{\zeta} E_\zeta \notin L^\infty(\R)$, we can have a similar proof of Part (2) for an appropriate choice of $\psi$.
	\end{proof}
	
	\begin{lemma}\label{lemma4.5}
		Let $T$ be a tempered distribution on $\R^d$ such that the distributional support of $\hat{T}$ is $\{ \zeta^{(1)}, \zeta^{(2)}, \ldots, \zeta^{(k)} \}$. Then $T = \ds \sum_{j=1}^{k} P_j(\partial_{\zeta}) E_{\zeta}|_{\zeta=\zeta^{(j)}}$ for some polynomials $P_j, 1 \leq j \leq k.$ In particular, if $T$ is given by a measurable function $f \in L^{p}(\R^d)$, then $f = c_1 E_{\zeta^{(1)}}+ c_2 E_{\zeta^{(2)}}+ \cdots + c_k E_{\zeta^{(k)}}$ for some constants $c_1, c_2, \ldots, c_k$ when $p=\infty$, and $f=0$ when $1\leq p <\infty$.
		\end{lemma}
		
		\begin{proof}
			Since the distributional support of $\hat{T}$ is the finite set $\{ \zeta^{(1)}, \zeta^{(2)}, \ldots, \zeta^{(k)} \}$, it is already established that there exist polynomials $P_j$ for $1 \leq j \leq k$ such that $$ \widehat{T} = \ds \sum_{j=1}^{k} P_j(\partial_{\zeta}) \delta_{\zeta^{(j)}} $$ where $ \delta_{\zeta^{(j)}}$ represents the dirac mass at $\zeta^{(j)}$. Now if we define $S = \ds \sum_{j=1}^{k} P_j(\partial_{\zeta}) E_{\zeta}|_{\zeta=\zeta^{(j)}}$, it is straight forward to see that $ \widehat{S} = \ds \sum_{j=1}^{k} P_j(\partial_{\zeta}) \delta_{\zeta^{(j)}}$. Invoking the injectivity of Fourier transform of tempered distributions on $\R^d$, we obtain that $T = \ds \sum_{j=1}^{k} P_j(\partial_{\zeta}) E_{\zeta}|_{\zeta=\zeta^{(j)}}$. If $T$ is given by a measurable function $f \in L^{p}(\R^d)$, then a direct application of Lemma \ref{lemma4.4} yields that $f = c_1 E_{\zeta^{(1)}}+ c_2 E_{\zeta^{(2)}}+ \ldots + c_k E_{\zeta^{(k)}}$ for some constants $c_1, c_2, \ldots, c_k$ when $p=\infty$, and $f=0$, when $1\leq p <\infty$.
		\end{proof}
		Next we shall provide an important proposition which is essentially required for us in deriving the proof of Theorem \ref{theorem4.3}. For the sake of brevity, we omit its proof and urge the reader to see \cite[Proposition 5.0.8]{muna} where an equivalent version is already established using elementary techniques from linear algebra.
		\begin{proposition}\label{proposition4.6}
			Let $f$ be a measurable function on $\R^d$ which can be written as a finite sum $f = f_1 + f_2 + \cdots +f_n$, where for some linear operator $\Theta$, $\Theta f_i = \ap_i f_i$, for $i = 1, \ldots, n$ with $\ap_1, \ap_2, \ldots, \ap_n \in \C$ distinct. Then \begin{equation}\label{eq:theta}
				(\Theta-\ap_1 I) (\Theta-\ap_2 I) \cdots (\Theta-\ap_n I) f= 0. 
			\end{equation} 
			Conversely, if a measurable function $f$ on $\R^d$ satisfies \eqref{eq:theta} for distinct $\ap_1, \ap_2, \ldots, \ap_n \in \C $, then $f$ can be uniquely written as a finite sum of eigenfunctions of $\Theta$ corresponding to the eigenvalues $\ap_1, \ap_2, \ldots, \ap_n $.
		\end{proposition}
		
		Now that we have stated all the required intermediate results, let us prove Theorem \ref{theorem4.3}.
\begin{proof}[Proof of Theorem \ref{theorem4.3}] Owing to Theorem 4.1(iii), we obtain that the distributional support of $ \widehat{f_{0}}$ is contained in the set $E= \{\xi \in \R^d| |m(\xi)| = |A| \}$. Suppose $E= \{ \zeta^{(1)}, \zeta^{(2)}, \ldots, \zeta^{(k)} \}$, then it follows from Lemma \ref{lemma4.5}  that $f_0=  c_1 E_{\zeta^{(1)}}+ c_2 E_{\zeta^{(2)}}+ \ldots + c_k E_{\zeta^{(k)}}$ for some constants $c_1, c_2, \ldots, c_k.$ Now combining the terms $E_{\zeta^{(j)}}$'s that share the same eigenvalue $m(\xi)$ under the operator $\Theta$,we may express $ f_0 = g_1+ g_2 + \ldots + g_r$, where each $g_i$ is an eigenfunction of $\Theta$ corresponding to a distinct eigenvalue $A_i$.  Hence we obtain $(\Theta-A_1 I)(\Theta-A_2 I) \ldots (\Theta-A_r I)f_0 = 0$. In view of  Proposition \ref{proposition4.6}, we have the required conclusion. The last part of the assertion is immediate.
\end{proof}
We restrict ourselves to radial tempered distributions that satisfies the Strichartz recursion and boundedness with respect to the Schwartz seminorms for the rest of this section.  First we shall prove an equivalent versions of Theorem \ref{theorem4.1} for the radial tempered distributions, that determines their support under the spherical Fourier transform.  In what follows let $m_0(\la)= Rm(\la e_1)$, where $R$ denotes the radialization operator.
	\begin{theorem}\label{theorem4.7}
	Let $\Theta$ be a multiplier on $\mathcal{S}(\mathbb{R}^d)$ with symbol $m_0(\la) \in \mathcal{C}^{\infty}(\R)$, 
	Suppose $\{T_k\}_{k \in \Z}$ be a doubly infinite sequence of radial tempered distributions on $\mathbb{R}^d$ such that for all $k \in \mathbb{Z}$, $|\langle T_k, \psi \rangle| \leq M \gamma(\psi)$ for a fixed seminorm $\gamma$ on $\mathcal{S}(\mathbb{R}^d)$ and a constant $M>0$.
	Let $A \in \mathbb{C}$. Then the following conclusions hold.
	\begin{enumerate}[(i)]
		\item If for all $k \in \mathbb{Z}^+$, $\Theta T_k = A T_{k-1}$, then Supp $\mathcal{H}{T}_0 \subseteq \{ \la \in \mathbb{R} \mid |m_0(\la)| \leq |A| \}$.
		\item If for all $k \in \mathbb{Z}^-$, $\Theta T_{k-1} = A T_k$, then Supp $ \mathcal{H}{T}_0 \subseteq \{ \la \in \mathbb{R} \mid |m_0(\la)| \geq |A| \}$.
		\item If for all $k \in \mathbb{Z}$, $\Theta T_k = A T_{k+1}$, then Supp $ \mathcal{H}{T}_0 \subseteq \{ \la \in \mathbb{R} \mid |m_0(\la)| = |A| \}$.
	\end{enumerate}
\end{theorem}
 
 The proof of the above theorem can be obtained after suitably modifying the proof of Theorem \ref{theorem4.1} using the spherical Fourier transform. Next, we shall see a result (analogous version of Lemma \ref{lemma4.5}) that provides the structure of a radial tempered distribution whose spherical Fourier transform is supported on a finite set.
\begin{lemma}\label{lemma4.8}
	Let $T$ be a radial tempered distribution on $\R^d$ such that the distributional support of $\mathcal{H}(T) $ is $\{\la_{1},\la_{2}, \dots, \la_{k}\}$, then $T= P_1(\pa_{\la})\phi_\la \mid_{\la=\la_{1}} + P_2(\pa_{\la})\phi_\la \mid_{\la=\la_{2}}+ \cdots + P_k(\pa_{\la})\phi_\la \mid_{\la=\la_{k}}$ for some polynomials $P_1, P_2, \ldots, P_k$. 
\end{lemma}

The proof of the above lemma follows by the similar arguments as in the proof of Lemma \ref{lemma4.5}.

\section*{Acknowledgment}
	The first author would like to thank UGC-CSIR, India for providing the fellowship for the completion of this work.

\end{document}